\documentclass[preprint,10pt]{elsarticle}

\usepackage{amssymb}
\usepackage{amsfonts}
\usepackage{amsmath,xfrac}
\usepackage{mathtools}   %
\usepackage[T1]{fontenc}
\usepackage{lmodern}
\usepackage{graphicx,amsthm,amsmath,ipn_tools}
\usepackage[dvipsnames]{xcolor}
\usepackage{tikz}
\usepackage{pgfplots}
\usepackage[margin=3.0cm]{geometry}
\usepackage{hyperref}
\hypersetup{
colorlinks=true,
linkcolor=blue,
filecolor=magenta,
urlcolor=cyan,
citecolor=gray,
}

\newtheorem{theorem}{Theorem}[section]
\newtheorem{lemma}{Lemma}[section]
\newtheorem{proposition}{Proposition}[section]
\newtheorem{corollary}{Corollary}[section]
\newtheorem{definition}{Definition}[section]

\newtheorem{remark}{Remark}[section]

\pgfplotsset{width=7.5cm,compat=1.12}
\usetikzlibrary{patterns}

\definecolor{carmine}{rgb}{0.59, 0.0, 0.09}
\definecolor{burntorange}{rgb}{0.8, 0.33, 0.0}

\usepackage{pgfplotstable}

\journal{Computer Methods in Applied Mechanics and Engineering}

\newcommand{\fluxseminorm}[1]{\lvert #1\rvert_{h,\sigma}}
\begin{document}

\begin{frontmatter}

\newdefinition{asp}{Assumption}
\newdefinition{rmk}{Remark}
\newdefinition{example}{Example}

\title{A hybridization-derived interior penalty Nitsche method for heterogeneous and anisotropic diffusion}

\affiliation[reunion]{organization={University of La R\'eunion, ESIROI},
city={Saint-Pierre}, postcode={97455}, state={La Réunion}, country={France} }

\author[reunion]{Vincent Fontaine}
\ead{vincent.fontaine@univ-reunion.fr}

\begin{abstract}
We develop an interior penalty Nitsche finite element formulation for heterogeneous and anisotropic diffusion problems, which is derived from an interior-penalty stabilized Lagrange multiplier method. A distinctive feature of the derived method is that it simultaneously combines two popular stabilization strategies: (i) a boundary penalty mechanism similar to Nitsche's methods and (ii) an interior penalty mechanism as for Continuous Interior Penalty (CIP) methods. Furthermore, the normal component of the numerical flux is single-valued across the mesh skeleton and available in closed form.
We establish the equivalence between the hybridized Lagrange multiplier formulation and its condensed Interior Penalty Nitsche (IPN) form, and we deliver the stability and convergence analysis of the latter. For sufficiently regular solutions, we establish energy-norm error estimates for the three symmetry variants. Under a piecewise elliptic regularity assumption on the adjoint problem, the $\LL{2}$ estimate is optimal for the symmetric variant and exhibits a half-order loss for the incomplete and nonsymmetric variants; its constant may depend on the diffusion coefficients through the adjoint regularity constant. This single-valued numerical normal flux, available in closed form, provides a starting point for locally conservative $\boldsymbol{\mathrm{H}}(\mathrm{div})$-conforming flux reconstruction, left for future work.
Numerical experiments validate the theoretical findings. %
\end{abstract}


\end{frontmatter}

\section{Introduction}
\label{Intro}
Using Lagrange multipliers on the domain boundary to handle non-homogeneous Dirichlet conditions was initially introduced and analyzed by I. Babu\v{s}ka \cite{Babuska1973}. However, the polynomial approximation spaces used for the primal variable and Lagrange-multipliers cannot be chosen independently, and they have to fulfill the so-called Babu\v{s}ka--Brezzi (BB) condition to ensure stability and optimal convergence. The construction of such a stable pair of finite element spaces was studied in detail by Pitk\"aranta \cite{Pitkaranta1979}, indicating that it is a delicate question. Indeed, both variables are strongly coupled since the multiplier can be physically interpreted as the outward normal boundary flux. Barbosa and Hughes \cite{BarbosaHughes1991} explored an original way to circumvent this issue and restore stability, notably by adding residual stabilization terms into the original discrete bilinear form. Later, Stenberg \cite{Stenberg1995} established some close connections between the stabilized formulation of Barbosa--Hughes and the original symmetric Nitsche method \cite{Nitsche1971} (without any multipliers) using only boundary penalty mechanisms instead to weakly impose Dirichlet conditions. A non-symmetric version of the Nitsche method, coercive for any positive value of the penalty parameter, was subsequently proposed and analyzed by Freund and Stenberg \cite{FreundStenberg1995}.

Here, we explore a different route where the multiplier now acts as the trace of the primal variable on the whole mesh skeleton, not only on the boundary interfaces. This complete hybridization is considered here to enforce additional regularity requirements by construction, such as flux continuity. The choice of approximation spaces becomes natural, since the stabilization removes the need for a delicate inf-sup compatible pairing. The stability of the fully hybridized Lagrange multiplier method is ensured by adding consistent terms that penalize the difference between the discrete primal variable and its trace on the mesh skeleton. The idea of blending a hybridized Lagrange-multiplier formulation with a penalty stabilization can be traced back to the penalty-hybrid finite element method of Fix, Liang and Lee \cite{FixLiangLee1982} for the Laplace operator, an early precursor of the hybridized interior penalty (H-IP) framework~\cite{EtangsaleFRWA2021,EtangsaleCAMWA2022}. These constructions can be viewed in the general setting of the earlier unified hybridization theory of Cockburn, Gopalakrishnan and Lazarov \cite{CGLSIAM09}, which encompasses discontinuous Galerkin, mixed and continuous Galerkin methods within a single setting, building on the unified analysis of discontinuous Galerkin methods of Arnold, Brezzi, Cockburn and Marini \cite{ABCMUnified}.

Within this framework, the embedded discontinuous Galerkin (EDG) method of Cockburn \al \cite{CockburnEDG2009} keeps the primal variable discontinuous but constrains the skeletal trace to be continuous, thereby reducing the global trace system to a continuous Galerkin stencil; the present IPN method adopts the transposed configuration, with a continuous ($H^1$-conforming) primal variable and a discontinuous, face-wise broken skeletal trace. A complementary contrast is offered by the Hybrid High-Order (HHO) setting, where Cascavita, Chouly and Ern~\cite{CascavitaChoulyErn2020} combined a hybrid skeletal discretization with Nitsche's method to impose Dirichlet and Signorini conditions weakly. There, the cell--face stabilization governs the bulk discretization, whereas the Nitsche terms act on the boundary, and static condensation proceeds in the canonical direction by eliminating the cell unknowns and retaining a global skeletal system. The present construction performs the complementary reduction: the skeletal trace is eliminated in closed form, the global problem is carried by the $H^1$-conforming primal field, and the boundary and interior penalty mechanisms are jointly inherited from the same reduced hybrid structure.

A practical advantage of these Lagrange multiplier methods is that they are eligible for static condensation, which mitigates the cost of the additional skeletal unknowns. However, the Schur complement strategy presents an additional computational complexity since it requires a set of local operations on the mesh skeleton, such as matrix assembly and inversion. An alternative way consists of algebraically removing multipliers from the set of discrete equations, in the spirit of the elimination connecting stabilized multiplier and Nitsche formulations established by Stenberg \cite{Stenberg1995}. This elimination is performed here by observing that discrete numerical traces are expressed solely in terms of the primal variable. Thus, the resulting $H^1$-conforming finite element method combines (i) a boundary penalty mechanism to impose Dirichlet conditions weakly as for Nitsche's methods \cite{Nitsche1971} and (ii) an interior penalty mechanism enforcing the weak continuity of the normal component of the numerical flux across the mesh skeleton---a skeletal $\boldsymbol{\mathrm{H}}(\mathrm{div})$-compatibility condition---as for CIP methods \cite{DouglasCIP1976}. We call it the Interior Penalty Nitsche (IPN) method.

The present construction is the $H^1$-conforming counterpart of the broken UIP-DG formulation~\cite{EtangsaleUIPDG2026}. Conformity removes the interior mean weights, leaving the transmissibility-dependent scalar $\varrho$ in the flux-jump penalty. With strongly imposed homogeneous Dirichlet conditions, the reduced family contains CG, positively stabilized and sign-reversed flux-jump variants; the coercivity threshold relaxes, and all three variants attain the optimal $\LL{2}$ estimate under the stated piecewise adjoint regularity assumption. Sections~\ref{S:3} and~\ref{S:4} distinguish these properties from those of the formulation with Nitsche boundary terms.

Related primal interface formulations were derived from stabilized multipliers through a variational multiscale argument by Truster and Masud~\cite{TrusterMasud2014}, and extended to finite strains by Truster, Chen and Masud~\cite{TrusterChenMasud2015}. Their interface constructions provide relevant precedents for combining weighted fluxes, primal jumps and flux-jump terms. Here, global $H^1$-conformity removes interior primal jumps, while face-local elimination of the skeletal trace retains a transmissibility-weighted normal-flux-jump term. The contribution is the resulting conforming three-variant family, its explicit numerical traces, and the stability and error analysis of the interior mechanism together with the weak boundary treatment.

The contributions of this work are fourfold: (i) the construction of an $H^1$-conforming interior penalty Nitsche method by closed-form, face-local elimination of the skeletal trace from an interior-penalty stabilized hybrid formulation; (ii) the resulting single-valued skeletal normal flux, available in closed form; (iii) a complete stability and \textit{a priori} error analysis, in the energy and $\LL{2}$ norms, for the three symmetry variants, together with the conforming CG / positively stabilized / sign-reversed trichotomy; and (iv) a numerical study of the contrast robustness and of the role of the transmissibility weighting. The embedding of continuous Galerkin methods into hybridization frameworks is classical~\cite{CGLSIAM09}; what is distinctive in (i) is that the explicit trace elimination generates a one-parameter primal family carrying a transmissibility-weighted flux-jump stabilization.

The rest of the material is organized as follows. Section \ref{S:2} introduces the discrete notations and the trace operators. Section \ref{S:3} constructs the IPN method: it introduces an auxiliary hybridizable Nitsche formulation, eliminates the skeletal unknown in closed form, and characterizes the resulting numerical traces. Section \ref{S:4} carries out the stability and \textit{a priori} error analysis of the IPN method, delivering well-posedness and error estimates of the scalar variable in the energy- and $\LL{2}$-norms. Section \ref{S:5} reports numerical experiments assessing convergence, contrast sensitivity and the role of the penalty weighting. We briefly end with some concluding remarks and perspectives, including the locally conservative $\boldsymbol{\mathrm{H}}(\mathrm{div})$-flux correction. The extension to linear elasticity is developed in a separate companion manuscript.

\section{Preliminaries}
\label{S:2}

\subsection{Model problem}

Let $\domain\subset\R^d$ ($d\in\{2,3\}$) be a bounded polyhedral domain with Lipschitz boundary $\bnddom$. We consider the heterogeneous, anisotropic diffusion problem with homogeneous Dirichlet boundary conditions,
\begin{equation}
\label{model_problem}
-\dvg{(\bkap\grd{u})}=f\quad\textrm{in }\domain,\qquad u=0\quad\textrm{on }\bnddom,
\end{equation}
where $f\in\lebesgue{2}{\domain}$ is a source term and $\bkap:\domain\to\R^{d\times d}$ is a symmetric diffusion tensor, piecewise constant with respect to the mesh introduced below and possibly discontinuous across material interfaces. We assume that there exist constants $0<\kappa_{\min}\le\kappa_{\max}<\infty$ such that
\begin{equation}
\label{ellipticity}
\kappa_{\min}\abs{\xi}^2\le\xi^{\top}\bkap(x)\,\xi\le\kappa_{\max}\abs{\xi}^2,\qquad\forall\xi\in\R^d,\quad\textrm{for a.e. }x\in\domain.
\end{equation}
The ratio $\kappa_{\max}/\kappa_{\min}$, referred to as the \emph{contrast ratio}, may be arbitrarily large; the robustness of the analysis with respect to this regime is one of the motivations of the present work. By the Lax--Milgram lemma, problem~\eqref{model_problem} admits a unique weak solution $u\in V\eqbydef\hilbert{1}{\domain}{0}$ such that
\begin{equation}
    \label{weak_form}
    \dbilin{a}{u}{v}\eqbydef\inerprod{\bkap\grd u}{\grd v}{0}{\domain}=\inerprod{f}{v}{0}{\domain}=:\dlin{l}{v},\qquad\forall v\in V,
\end{equation}
where $\inerprod{\cdot}{\cdot}{0}{\domain}$ denotes the $\lebesgue{2}{\domain}$ inner product. In the context of Darcy flows, $u$ represents the potential and $\sg\eqbydef-\bkap\grd{u}$ the diffusive flux.

\subsection{Discrete setting}

\paragraph{Mesh and skeleton}
Let $\{\Th\}_{h>0}$ be a shape-regular family of conforming affine simplicial meshes of $\domain$, with global meshsize $h\eqbydef\max_{\e{}\in\Th}h_{\e{}}$ and local diameter $h_{\e{}}\eqbydef\diam{\e{}}$; we assume $h\le 1$ without loss of generality. For $d=2$ (resp.\ $d=3$), an \emph{interface} is an edge (resp.\ a face). The mesh skeleton is partitioned as $\Fh\eqbydef\Fhi\cup\Fhb$, where $\Fhi$ and $\Fhb$ collect the interior and boundary interfaces, respectively. For each element $\e{}\in\Th$, we denote by $\FE$ the set of its interfaces, by $\vect{n}_{\e{},\f{}}$ the unit outward normal to $\f{}\in\FE$, and we set
\[
\bnd{\Th}\eqbydef\bigcup_{\e{}\in\Th}\FE,
\qquad
\eta_{\e{}}\eqbydef\card{\FE},
\qquad
\eta_0\eqbydef\max_{\e{}\in\Th}\eta_{\e{}}.
\]
Note that $\bnd{\Th}$ is a multiset, each interior interface being counted twice (once from each side), whereas $\Fh$ counts each interface once. For a measurable set $X$, $\abs{X}$ denotes its $d$- or $(d-1)$-dimensional Lebesgue measure.

\paragraph{Functional spaces}
For a polyhedral domain $D\subset\R^d$, $\hilbert{s}{D}{}$ is the usual Sobolev space with norm $\norm{\cdot}{s}{D}$ and seminorm $\seminorm{\cdot}{s}{D}$; for $s=0$, $\hilbert{0}{D}{}=\lebesgue{2}{D}$ with inner product $\inerprod{\cdot}{\cdot}{0}{D}$, and on boundaries we write $\dualprod{\cdot}{\cdot}{0}{\bnd{D}}$ for the $\lebesgue{2}{\bnd{D}}$ inner product. The \emph{broken Sobolev space} over $\Th$ is
\[
\hilbert{s}{\Th}{}\eqbydef\{v\in\lebesgue{2}{\domain}:v\vert_{\e{}}\in\hilbert{s}{\e{}}{}\ \forall\e{}\in\Th\},
\]
equipped with the broken norm $\norm{\cdot}{s}{\Th}$ and the broken gradient $\grdh{}$ acting on $\hilbert{s}{\Th}{}$, $s\ge 1$. To handle interface integrals uniformly, we use the aggregated $\LL{2}$-inner products
\[
\dualprod{\cdot}{\cdot}{0}{\bnd{\Th}}\eqbydef\sum_{\e{}\in\Th}\sum_{\f{}\in\FE}\dualprod{\cdot}{\cdot}{0}{\f{}},
\qquad
\dualprod{\cdot}{\cdot}{0}{\Fh}\eqbydef\sum_{\f{}\in\Fh}\dualprod{\cdot}{\cdot}{0}{\f{}},
\]
with corresponding norms $\norm{\cdot}{0}{\bnd{\Th}}$ and $\norm{\cdot}{0}{\Fh}$. The diffusion tensor is piecewise constant, $\bkap\in[\Pk{0}{\Th}]^{d\times d}$; on each $\e{}\in\Th$, the largest and smallest eigenvalues of $\bkap_{\e{}}\eqbydef\bkap\vert_{\e{}}$ are denoted by $\overline{\kappa}_{\e{}}$ and $\underline{\kappa}_{\e{}}$, and we set $\overline{\kappa}_{0}\eqbydef\max_{\e{}\in\Th}\overline{\kappa}_{\e{}}$ and $\underline{\kappa}_{0}\eqbydef\min_{\e{}\in\Th}\underline{\kappa}_{\e{}}$.

\paragraph{Polynomial spaces and composite space}
Let $k\ge 1$, and let $\Pk{k}{X}$ denote the space of polynomials of total degree at most $k$ on a generic element or interface $X$. In contrast to the broken discontinuous Galerkin setting of~\cite{EtangsaleUIPDG2026}, the primal variable is here sought in the \emph{$H^1$-conforming} space
\begin{equation}
\label{spaces}
\Vh\eqbydef\{\dv\in\hilbert{1}{\domain}{}:\dv\vert_{\e{}}\in\Pk{k}{\e{}},\ \forall\e{}\in\Th\},
\end{equation}
whereas the Lagrange multiplier lives in the single-valued, face-wise broken trace space $\TVh{}\eqbydef\{\dtv\in\lebesgue{2}{\Fh}:\dtv\vert_{\f{}}\in\Pk{k}{\f{}},\ \forall\f{}\in\Fh\}$. The homogeneous Dirichlet condition is enforced at the skeletal level through $\TVh{0}\eqbydef\{\dtv\in\TVh{}:\dtv\vert_{\f{}}=0,\ \forall\f{}\in\Fhb\}$. We finally introduce the \emph{composite space} $\CVh{0}\eqbydef\Vh\times\TVh{0}$, whose generic element is a pair $\cdv\eqbydef(\dv,\dtv)$.

\subsection{Discrete trace operators}
\label{S:2.3}

To bridge the hybridized and primal formulations, we introduce three trace operators together with a \emph{skeletal-weighting function} $\vomega:\Fh\to\R^2$. For each interior interface $\f{}=\bnd{\e{1}}\cap\bnd{\e{2}}\in\Fhi$, we set $\vomega\vert_{\f{}}\eqbydef(\omega_1,\omega_2)$ with $\omega_i\eqbydef\omega\vert_{\e{i},\f{}}\ge 0$ and $\omega_1+\omega_2=1$ a.e.\ on $\f{}$; on $\f{}\in\Fhb$ we set $(\omega_1,\omega_2)\eqbydef(1,0)$.

\begin{definition}[Discrete trace operators]
\label{def_traces_operators}
Let $s\ge 1$. For an interior interface $\f{}=\bnd{\e{1}}\cap\bnd{\e{2}}\in\Fhi$, set $\vect{n}_i\eqbydef\vect{n}_{\e{i},\f{}}$, and for $\varphi\in\hilbert{s}{\Th}{}$ (resp.\ $\boldsymbol{\varphi}\in[\hilbert{s}{\Th}{}]^d$) write $\varphi_i\eqbydef\varphi\vert_{\e{i},\f{}}$ (resp.\ $\boldsymbol{\varphi}_i\eqbydef\boldsymbol{\varphi}\vert_{\e{i},\f{}}$).
\begin{itemize}
  \item \textbf{DG jump.} The normal jump $\jump{\cdot}$ acts on scalar and vector fields by
  \begin{equation}
    \label{sub:wjump}
    \jump{\varphi}\vert_{\f{}}\eqbydef
    \begin{cases}\varphi_1\vect{n}_1+\varphi_2\vect{n}_2 & \f{}\in\Fhi,\\ \varphi\vert_{\e{},\f{}}\,\vect{n}_{\e{},\f{}} & \f{}\in\Fhb,\end{cases}
    \qquad
    \jump{\boldsymbol{\varphi}}\vert_{\f{}}\eqbydef
    \begin{cases}\boldsymbol{\varphi}_1\cdot\vect{n}_1+\boldsymbol{\varphi}_2\cdot\vect{n}_2 & \f{}\in\Fhi,\\ \boldsymbol{\varphi}\vert_{\e{},\f{}}\cdot\vect{n}_{\e{},\f{}} & \f{}\in\Fhb,\end{cases}
  \end{equation}
  so that $\jump{\varphi}$ is vector-valued and $\jump{\boldsymbol{\varphi}}$ is scalar-valued.
  \item \textbf{Weighted DG mean.} Given $\vomega$, the weighted mean and its conjugate are
  \begin{equation}
    \label{sub:wmean}
    \wmean{\boldsymbol{\varphi}}{\vomega}\vert_{\f{}}\eqbydef
    \begin{cases}\omega_1\boldsymbol{\varphi}_1+\omega_2\boldsymbol{\varphi}_2 & \f{}\in\Fhi,\\ \boldsymbol{\varphi}\vert_{\e{},\f{}} & \f{}\in\Fhb,\end{cases}
    \qquad
    \cwmean{\boldsymbol{\varphi}}{\vomega}\vert_{\f{}}\eqbydef
    \begin{cases}\omega_2\boldsymbol{\varphi}_1+\omega_1\boldsymbol{\varphi}_2 & \f{}\in\Fhi,\\ \boldsymbol{\varphi}\vert_{\e{},\f{}} & \f{}\in\Fhb,\end{cases}
  \end{equation}
  and likewise for scalar fields; the choice $\omega_1=\omega_2=\shalf$ recovers the arithmetic mean $\mean{\cdot}$.
  \item \textbf{HDG local jump.} The HDG jump $\tjump{\cdot}:\CVh{0}\to[\lebesgue{2}{\bnd{\Th}}]^d$ measures, elementwise, the discrepancy between the primal field and its skeletal trace,
  \begin{equation}
    \label{jumpoperator}
    \tjump{\cdv}\vert_{\e{},\f{}}\eqbydef(\dv-\dtv)\vert_{\f{}}\,\vect{n}_{\e{},\f{}},
    \qquad\forall\e{}\in\Th,\ \forall\f{}\in\FE.
  \end{equation}
\end{itemize}
\end{definition}
Because $\dv\in\Vh\subset\hilbert{1}{\domain}{}$ has a single-valued trace and $\dtv$ is single-valued on $\Fh$, the scalar quantity $(\dv-\dtv)\vert_{\f{}}$ is single-valued across each interface; in particular,
\begin{equation}
\label{conforming_jump}
\jump{\dv}=\vect{0}\quad\textrm{on}\quad\Fhi,\qquad\forall\dv\in\Vh.
\end{equation}
This is a direct consequence of the $H^1$-conformity. It reduces the general boundary-to-skeleton identity of~\cite{EtangsaleUIPDG2026} to the following conforming form.
\begin{lemma}[Conforming identity relation]
\label{lemma:identity_rel}
Let $\vect{b}\in[\hilbert{1}{\Th}{}]^d$ and $\cdv=(\dv,\dtv)\in\CVh{0}$. Then
\begin{equation}
\label{magic3}
\dualprod{\vect{b}}{\tjump{\cdv}}{0}{\bnd{\Th}}
=\dualprod{\vect{b}\cdot\vect{n}}{\dv}{0}{\Fhb}
+\dualprod{\jump{\vect{b}}}{\dv-\dtv}{0}{\Fhi}.
\end{equation}
\end{lemma}
\begin{proof}
Expanding the multiset sum over $\bnd{\Th}$ and grouping the two contributions of each interior interface, we obtain
\begin{equation*}
\dualprod{\vect{b}}{\tjump{\cdv}}{0}{\bnd{\Th}}
=\sum_{\f{}\in\Fhi}\dualprod{\vect{b}_1\cdot\vect{n}_1+\vect{b}_2\cdot\vect{n}_2}{\dv-\dtv}{0}{\f{}}
+\sum_{\f{}\in\Fhb}\dualprod{\vect{b}\cdot\vect{n}}{\dv-\dtv}{0}{\f{}},
\end{equation*}
where the scalar factor $(\dv-\dtv)\vert_{\f{}}$ is single-valued on every $\f{}\in\Fhi$ ($\dv$ being $H^1$-conforming and $\dtv$ single-valued), so that it factors out of the two-sided sum, whose vector part is precisely $\jump{\vect{b}}$. On $\f{}\in\Fhb$, $\dtv=0$ since $\dtv\in\TVh{0}$. This is~\eqref{magic3}. Since the primal trace is single-valued, every admissible weighted average coincides with that trace on interior faces.
\end{proof}

\section{The interior penalty Nitsche method}
\label{S:3}

This section constructs the interior penalty Nitsche (IPN) method and characterizes its numerical traces. The construction proceeds through an auxiliary hybridizable Nitsche formulation whose skeletal unknown is statically condensed, under the sole assumption that $\tau_{\e{},\f{}}$ is a strictly positive constant on every element--face pair, into the primal IPN method posed on $\Vh$. From the same condensation, the trace of the primal variable and the numerical flux are recovered in closed form, as a local post-processing of the IPN solution $\du$; the resulting numerical normal flux is single-valued across the skeleton.

\subsection{The hybridizable Nitsche formulation}

This auxiliary formulation is built from the conservative structure of the model problem~\eqref{model_problem}. Multiplying it by a test function $\dv\in\Vh$, integrating by parts over each element $\e{}\in\Th$ and summing over the mesh, the exact solution satisfies the global balance
    \begin{equation}
        \label{globweak}
        -\inerprod{\sg}{\grdh\dv}{0}{\Th}+\dualprod{\tsg\cdot\vect{n}}{\dv}{0}{\bnd{\Th}}=\inerprod{f}{\dv}{0}{\Th},
    \end{equation}
    where $\sg\eqbydef-\bkap\grd{u}$ is the exact diffusive flux and $\tsg$ its trace on the mesh skeleton; as $\sg\in\boldsymbol{H}(\mathrm{div},\domain)$, its normal trace is single-valued, so that $\tsg\cdot\vect{n}=\sg\cdot\vect{n}$ across interfaces. Since $\dtv\in\TVh{0}$ vanishes on $\Fhb$, we enforce the weak flux-transmission requirement
    \begin{equation}
        \label{transcond}
    	\dualprod{\tsg\cdot\vect{n}}{\dtv}{0}{\bnd{\Th{}}}=0,
    \end{equation}
    which expresses the weak continuity of the normal flux $\tsg\cdot\vect{n}$ across $\Fhi$. Subtracting \eqref{transcond} from \eqref{globweak} leads to the continuous hybridized form
    \begin{equation}
    \label{finprob}
      -\inerprod{\sg}{\grdh\dv}{0}{\Th{}}+\dualprod{\tsg}{\tjump{\cdv}}{0}{\bnd{\Th{}}}=\inerprod{f}{\dv}{0}{\Th{}}.
    \end{equation}
    We now pass to the discrete level by substituting, in the continuous hybridized form~\eqref{finprob}, the exact triple $(u,\sg,\tsg)$ by its discrete counterpart $(\cdu,\dsg(\du),\dtsg(\cdu))$, where $\dsg(\du)\eqbydef-\bkap\grdh\du$ is the native diffusive flux on $\Th$ and $\dtsg(\cdu)$ its numerical trace on $\bnd{\Th}$. In contrast to the exact flux, the native discrete flux $\dsg(\du)$ is double-valued on $\Fhi$, since $\du\in\Vh$ is merely $H^1$-conforming and $\bkap$ may jump across interfaces, so $\dsg$ and $\dtsg$ no longer coincide and must be defined separately. The numerical trace is closed by the transmissibility-stabilized ansatz
    \begin{equation}
    \label{tracesg}        \dtsg(\cdu)\eqbydef\dsg(\du)+\tau\tjump{\cdu}\quad\textrm{on}\quad\bnd{\Th{}},
    \end{equation}
    where, at this stage, $\tau$ is only assumed to be double-valued on $\bnd{\Th}$ and, on every element--face pair $(\e{},\f{})$ with $\f{}\in\FE$, to reduce to a strictly positive constant $\tau_{\e{},\f{}}$. Positivity makes the skeletal equations locally solvable; the facewise constancy is what turns their solution into a pointwise closed form, as Lemma~\ref{lem:local_solvability} makes precise. Both are satisfied by the transmissibility-based expression~\eqref{penalty_term} specified later for the stability analysis, $\bkap_{\e{}}$, $\vect{n}_{\e{},\f{}}$, $\abs{\f{}}$ and $\abs{\e{}}$ being constant.
\begin{definition}[Hybridizable Nitsche method]
\label{def:HNG}
The auxiliary hybridizable Nitsche formulation reads: find $\cdu\eqbydef(\du,\dtu)\in\CVh{0}$ such that
\begin{equation}
\label{HN_discrete_problem}
\dbilin{\boldsymbol{a}_h}{\cdu}{\cdv} = \dlin{\boldsymbol{l}_h}{\cdv}\eqbydef\inerprod{f}{\dv}{0}{\Th},\qquad\forall\cdv\eqbydef(\dv,\dtv)\in\CVh{0},
\end{equation}
where the bilinear form, obtained by adding to~\eqref{finprob} a symmetrizing term weighted by $\epsilon\in\{0,\pm1\}$, reads
\begin{equation}
    \label{bilinear-HN}
    \dbilin{\boldsymbol{a}_h}{\cdu}{\cdv}\eqbydef\inerprod{-\dsg(\du)}{\grdh\dv}{0}{\Th{}}+\dualprod{\dtsg(\cdu)}{\tjump{\cdv}}{0}{\bnd{\Th{}}}+\epsilon\dualprod{\dsg(\dv)}{\tjump{\cdu}}{0}{\bnd{\Th{}}}.
\end{equation}
The last term is an adjoint-consistency term that controls the symmetry of the bilinear form, as in classical interior penalty discontinuous Galerkin methods; accordingly, $\boldsymbol{a}_h$ is symmetric if $\epsilon=1$ and non-symmetric otherwise.
\end{definition}

The next result shows that the skeletal unknown is determined by the primal one, and identifies what each of the two assumptions on $\tau$ contributes.
\begin{lemma}[Local solvability and closed-form elimination]
\label{lem:local_solvability}
Let $\tau>0$. Then, for every $\du\in\Vh$, the skeletal component $\dtu\in\TVh{0}$ is uniquely determined by \eqref{HN_discrete_problem}, through independent interface-local problems. If, in addition, $\tau_{\e{},\f{}}$ is constant on every element--face pair, then $\dtu$ is given in closed form, pointwise on each interface, in terms of $\du$. The elimination of $\dtu$ is therefore performed analytically, without assembling or inverting any skeletal system.
\end{lemma}
\begin{proof}
Selecting $\cdv\eqbydef(0,\dtv)$ in \eqref{HN_discrete_problem} gives $\dualprod{\dtsg(\cdu)\cdot\vect{n}}{\dtv}{0}{\bnd{\Th}}=0$ for all $\dtv\in\TVh{0}$. Substituting the flux ansatz \eqref{tracesg} yields
\[
\dualprod{\tau\,\dtu}{\dtv}{0}{\bnd{\Th}}=\dualprod{\dsg(\du)\cdot\vect{n}+\tau\,\du}{\dtv}{0}{\bnd{\Th}},\qquad\forall\dtv\in\TVh{0}.
\]
Since $\TVh{0}$ is broken face by face, this relation decouples into independent interface problems, each governed on $\f{}$ by the form $\dualprod{\tau\,\cdot}{\cdot}{0}{\f{}}$, which is symmetric positive definite on $\Pk{k}{\f{}}$ as soon as $\tau>0$; each is therefore uniquely solvable, and so is $\dtu$. Positivity alone, however, only yields $\dtu$ as the solution of a $\tau$-weighted mass problem: the quotient of the right-hand side by $\tau_1+\tau_2$ need not belong to $\Pk{k}{\f{}}$. When $\tau_{\e{},\f{}}$ is constant on every element--face pair, that quotient is a face polynomial and the interface problems are solved pointwise, which gives the explicit formula of Lemma~\ref{lemma_Trace} below and proves the claim.
\end{proof}
The closed form characterizes the numerical traces $\dtu$ and $\dtsg$ on $\Fh$ solely in terms of $\du$.

\subsection{Inherited numerical traces}

By the closed-form elimination of Lemma~\ref{lem:local_solvability}, the skeletal quantities $\dtu$ and $\dtsg$ are determined explicitly by the primal variable $\du$ alone. These relations both drive the reduction leading to the IPN method (Section~\ref{S:3.3}) and provide, once $\du$ is known, a local recovery on the skeleton of the numerical flux $\dtsg$, whose normal component is single-valued.
\begin{lemma}%
\label{lemma_Trace}
Let $\tau_{\e{},\f{}}$ be a strictly positive constant on every element--face pair and let $\cdu\eqbydef(\du,\dtu)\in\CVh{0}$ be a solution of the discrete problem \eqref{HN_discrete_problem}. Then, the flux-transmission requirement \eqref{transcond} implies that $(\dtu,\dtsg\cdot\vect{n})$ are single-valued on the mesh skeleton, \ie
\begin{subequations}
\label{traces_def}
\begin{align}
    \label{traces_def_u}
    \dtu&=\begin{cases}
    \du+\varrho\jump{-\bkap\grdh\du}{} & \textrm{on}\quad\Fhi,\\
    0 & \textrm{on}\quad\Fhb,
    \end{cases}\\
    \label{traces_def_sg}
    \dtsg\cdot\vect{n} &=\begin{cases}
    \cwmean{-\bkap\grdh\du}{\vomega}\cdot\vect{n} & \textrm{on}\quad\Fhi,\\
    -\bkap\grdh\du\cdot\vect{n} + \tau\du & \textrm{on}\quad\Fhb,
    \end{cases}
\end{align}
\end{subequations}
where $\vect{n}$ denotes either of the two unit normals to $\f{}$ on $\Fhi$, the choice being immaterial since both sides of~\eqref{traces_def_sg} change sign together, and the unit outward normal on $\Fhb$; $\vomega$ and $\varrho$ are both $\tau$-dependent valued functions defined as follows; \ie for all $\f{}\eqbydef\bnd{\e{1}}\cap\bnd{\e{2}}\in\Fhi$ then,
\begin{align}               
\vomega\vert_{\f{}}\eqbydef(\omega_{\e{1},\f{}},\omega_{\e{2},\f{}}),\quad\textrm{and}\quad\varrho\vert_{\f{}}\eqbydef\dfrac{1}{\tau_{\e{1},\f{}}+\tau_{\e{2},\f{}}},
\end{align}
where $\omega_{i}\eqbydef\varrho\vert_{\f{}}\,\tau_{\e{i},\f{}}$ with $i=1,2$. For all $\f{}\eqbydef\bnd{\e{}}\cap\bnd{\domain}\in\Fhb$, we set by convention that $\varrho\vert_{\f{}}\eqbydef\tau^{-1}_{\e{},\f{}}$ and hence $\vomega\vert_{\f{}}\eqbydef(1,0)$.
\end{lemma}
\begin{proof}
 Selecting $\cdv\eqbydef(0,\dtv)$, the discrete problem \eqref{HN_discrete_problem} reduces to $\dualprod{\dtsg(\cdu)\cdot\vect{n}}{\dtv}{0}{\bnd{\Th{}}}=0$ for all $\dtv\in\TVh{0}$, that is, $\dualprod{\jump{\dtsg(\cdu)}}{\dtv}{0}{\Fhi}=0$. On every interior interface $\f{}\in\Fhi$ the discrete flux jump is a face polynomial, $\jump{\dtsg(\cdu)}\vert_{\f{}}\in\Pk{k}{\f{}}$, since $\dsg(\du)\cdot\vect{n}\in\Pk{k-1}{\f{}}$ and, $\tau_{\e{i},\f{}}$ being constant, $\tau\,\tjump{\cdu}\cdot\vect{n}\in\Pk{k}{\f{}}$. As $\TVh{0}$ spans $\Pk{k}{\f{}}$ face by face, the choice $\dtv\vert_{\f{}}\eqbydef\jump{\dtsg(\cdu)}\vert_{\f{}}$ is admissible and yields $\norm{\jump{\dtsg(\cdu)}}{0}{\f{}}^2=0$, hence $\jump{\dtsg(\cdu)}\vert_{\f{}}=0$ in the $\LL{2}(\f{})$ sense. Imposing this flux continuity and isolating $\dtu$ gives, on each $\f{}=\bnd{\e{1}}\cap\bnd{\e{2}}\in\Fhi$ with $\vect{n}_1=-\vect{n}_2$,
 \begin{equation*}
 \dtu\vert_{\f{}}=\du+\dfrac{\dsg(\du)_{1}\cdot\vect{n}_{1}+\dsg(\du)_{2}\cdot\vect{n}_{2}}{\tau_{\e{1},\f{}}+\tau_{\e{2},\f{}}},\qquad\vect{n}_{2}=-\vect{n}_{1},
 \end{equation*}
 where the subscript $i$ denotes the restriction to $\e{i}$ as in Definition~\ref{def_traces_operators}; this is the compact form \eqref{traces_def_u}. The relation \eqref{traces_def_sg} then follows by substituting \eqref{traces_def_u} into \eqref{tracesg}.

\end{proof}
\subsection{Closed-form elimination: the IPN method}
\label{S:3.3}

\begin{theorem}[Closed-form reduction to the IPN method]
\label{thm:NCIP_method}
Assume that $\tau_{\e{},\f{}}$ is a strictly positive constant on every element--face pair. Substituting the closed-form trace $\dtu=\dtu(\du)$ of Lemma~\ref{lemma_Trace} into the hybridized bilinear form~\eqref{bilinear-HN} reduces the auxiliary hybridized problem~\eqref{HN_discrete_problem} to the following primal discrete problem, \ie %
\begin{equation}
    \label{NCIP_formulation}
    \begin{cases}
        \textrm{Find }\du\in\Vh\textrm{ such that,}\\
    \dbilin{a_{h}}{\du}{\dv}=\dlin{l}{\dv},\quad\forall\dv\in\Vh{},
    \end{cases}
\end{equation}
where $l(v_h) = (f,v_h)_{0,\mathcal{T}_h}$ and $a_h:\Vh\times\Vh\rightarrow\R$ is given as follows, 
\begin{equation}
    \label{bilinear_NCIP}
    \dbilin{a_h}{\du}{\dv}\eqbydef\dbilin{a}{\du}{\dv}+\epsilon\dbilin{s_{h}}{\du}{\dv}+\dbilin{\eta_{h}}{\du}{\dv}.
\end{equation}
Here, the second and third terms are respectively called the interior- and Nitsche's boundary-penalty mechanisms, \ie  
\begin{align*}
    \dbilin{s_{h}}{\du}{\dv}\eqbydef & -\dualprod{\varrho\jump{\bkap\grdh{\du}}}{\jump{\bkap\grdh{\dv}}}{0}{\Fhi},\\
    \dbilin{\eta_{h}}{\du}{\dv}\eqbydef & -\dualprod{\bkap\grdh\du\cdot\vect{n}}{\dv}{0}{\Fhb}-\epsilon\dualprod{\bkap\grdh\dv\cdot\vect{n}}{\du}{0}{\Fhb}+\dualprod{\tau\du}{\dv}{0}{\Fhb},
    \end{align*}
where $\varrho$ is defined in Lemma \ref{lemma_Trace}. In particular, the reduction is performed analytically and requires no assembled skeletal Schur complement; conversely, the trace $\dtu$ and the numerical flux $\dtsg$ are recovered from $\du$ in closed form through~\eqref{traces_def}.
\end{theorem}
\begin{proof}
We insert the closed-form trace $\dtu=\dtu(\du)$ of Lemma~\ref{lemma_Trace} into the flux form~\eqref{bilinear-HN}. By construction it enforces the flux continuity $\jump{\dtsg(\du)}\vert_{\f{}}=0$ for all $\f{}\in\Fhi$, while $\dtv\vert_{\f{}}=0$ for all $\f{}\in\Fhb$; hence Lemma~\ref{lemma:identity_rel}, applied with $\vect{b}\eqbydef\dtsg(\du)$, gives
\begin{equation}
\label{consistency_decomp}
\dualprod{\dtsg(\cdu)}{\tjump{\cdv}}{0}{\bnd{\Th{}}} =\dualprod{\jump{\dtsg(\du)}}{\dv-\dtv}{0}{\Fh}=\dualprod{\dtsg(\du)\cdot\vect{n}}{\dv}{0}{\Fhb}.
\end{equation}
Applying Lemma~\ref{lemma:identity_rel} once more, now with $\vect{b}\eqbydef\dsg(\dv)$, and inserting the closed-form trace~\eqref{traces_def_u}, the symmetry term reads 
    \begin{equation}
        \label{symmetry_decomp}
        \dualprod{\dsg(\dv)}{\tjump{\cdu}}{0}{\bnd{\Th{}}}=\dualprod{\jump{\dsg(\dv)}}{\du-\dtu}{0}{\Fh}
        =\dualprod{\dsg(\dv)\cdot\vect{n}}{\du}{0}{\Fhb}-\dualprod{\varrho\jump{\dsg(\du)}}{\jump{\dsg(\dv)}}{0}{\Fhi}.
    \end{equation}
By combining \eqref{consistency_decomp}-\eqref{symmetry_decomp} and identifying each term, one obtains
\begin{align*}
    \dbilin{s_{h}}{\du}{\dv}\eqbydef & -\dualprod{\varrho\jump{\dsg(\du)}}{\jump{\dsg(\dv)}}{0}{\Fhi},\\
    \dbilin{\eta_{h}}{\du}{\dv}\eqbydef &\dualprod{\dtsg(\du)\cdot\vect{n}}{\dv}{0}{\Fhb}+\epsilon\dualprod{\dsg(\dv)\cdot\vect{n}}{\du}{0}{\Fhb}.
    \end{align*}
This concludes the proof.
\end{proof}
According to the symmetrization parameter $\epsilon$, the IPN method~\eqref{NCIP_formulation} comes in three variants: the symmetric (S-IPN, $\epsilon=1$), non-symmetric (N-IPN, $\epsilon=-1$), and incomplete (I-IPN, $\epsilon=0$) schemes. If $\epsilon\neq 0$, the IPN method augments the corresponding standard Nitsche method by the interior penalty term $\epsilon s_h$, which acts on the jump of the normal diffusive flux $\jump{\bkap\grdh\du}$ across interior faces, in the spirit of CIP methods. Since $s_h$ is negative semi-definite, the genuinely penalizing flux-jump contribution corresponds to $\epsilon=-1$ (N-IPN), whereas the symmetric variant $\epsilon=+1$ (S-IPN) carries a sign-reversed, anti-penalizing term. In terms of cost, the IPN method keeps exactly the unknowns of the continuous Galerkin method. As for CIP-type stabilizations~\cite{BurmanHansbo2004}, the interior penalty $s_h$ couples the gradients of neighboring elements, and thereby extends the CG stencil from the element patch of a degree of freedom to the \emph{face-neighbours} of that patch (Figure~\ref{fig:stencil}). The stencil thus remains local, the number of couplings per unknown being bounded independently of $h$. The incomplete variant I-IPN ($\epsilon=0$), for which the interior mechanism is absent, retains the continuous Galerkin stencil exactly.
\begin{figure}[!ht]
\centering
\includegraphics[width=0.50\textwidth]{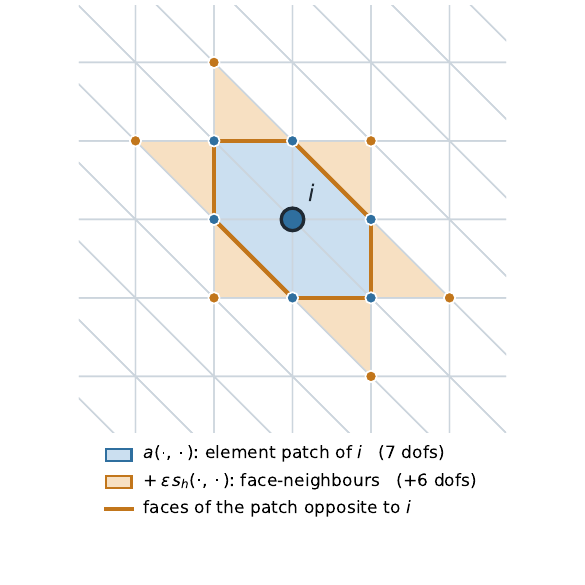}
\caption{Stencil of the IPN method at an interior degree of freedom $i$ (degree $k=1$). In blue, the element patch of $i$, which carries the volume form $a$; in orange, the face-neighbours of that patch, reached by the interior penalty $s_h$ alone through the faces of the patch opposite to $i$ (thick orange), on which $\bkap\grd\varphi_i$ does not vanish. The couplings displayed are structural, \ie determined by the overlap of the supports.}
\label{fig:stencil}
\end{figure}
\begin{remark}[Analogy with the Schur complement]
\label{remark2}
Let $\mathrm{U}_h$ denote the vector of primal degrees of freedom, and use the usual block notation for the hybrid system. Since the skeletal trace space is broken facewise, the skeletal block is block diagonal and can be eliminated face by face, giving
\begin{equation}
    \label{schur_complement}
    [\mathrm{A}_{uu}-\mathrm{A}_{u\hat{u}}\mathrm{A}^{-1}_{\hat{u}\hat{u}}\mathrm{A}_{\hat{u}u}]\mathrm{U}_h=\mathrm{F}.
\end{equation}
This Schur complement is precisely the matrix representation of the reduced IPN problem~\eqref{NCIP_formulation}. The closed-form trace formula of Theorem~\ref{thm:NCIP_method} performs this reduction analytically, avoiding the explicit assembly and inversion of the skeletal block.
\end{remark}
The method is boundary-nonconforming, in the sense that $\Vh\not\subset\hilbert{1}{\domain}{0}$ although the primal space is $H^1$-conforming ($\Vh\subset\hilbert{1}{\domain}{}$): the Dirichlet condition is imposed weakly through the Nitsche boundary mechanism $\eta_{h}$~\cite{Nitsche1971,Stenberg1995}. As the volume form $a$ and the interior penalty $s_{h}$ are both unconditionally symmetric, $\eta_{h}$ is the sole source of asymmetry of $a_h$; in particular, $a_h$ is symmetric if and only if $\epsilon=1$.
A symmetrized, conforming counterpart of the IPN method is obtained by restricting the primal variable to the tighter subspace $V_{h,0}\eqbydef\Vh\cap\hilbert{1}{\domain}{0}$, whose functions vanish on the domain boundary.
\begin{proposition}[The conforming $\HH{1}_{0}$ variant]
\label{prop:conforming}
Seek $\du\in V_{h,0}$ such that
\begin{equation}
    \label{CCIP_formulation}
    \dbilin{a^{\mathfrak{c}}_{h}}{\du}{\dv}=\dlin{l}{\dv},\quad\forall\dv\in V_{h,0},
    \qquad a^{\mathfrak{c}}_{h}\eqbydef a+\epsilon\,s_{h}.
\end{equation}
Then the Nitsche boundary term $\eta_{h}$ vanishes, so that $a^{\mathfrak c}_h=a+\epsilon\,s_h$ reads
\begin{equation}
\label{CCIP_bilinear}
\dbilin{a^{\mathfrak c}_h}{\du}{\dv}=\inerprod{\bkap\grdh\du}{\grdh\dv}{0}{\Th}-\epsilon\dualprod{\varrho\jump{\bkap\grdh\du}}{\jump{\bkap\grdh\dv}}{0}{\Fhi}.
\end{equation}
The form $a^{\mathfrak c}_h$ is symmetric for every $\epsilon\in\{0,\pm1\}$, and the three variants form a single transmissibility-weighted flux-jump family parametrized by $\epsilon$: the standard continuous Galerkin (CG) method ($\epsilon=0$), for which the interior mechanism is simply absent; a \emph{positively stabilized} flux-jump CIP \emph{à la} Douglas--Dupont \cite{DouglasCIP1976,BurmanHansbo2004} ($\epsilon=-1$); and a \emph{sign-reversed} flux-jump variant in the spirit of the variational multiscale stabilization of Codina \al \cite{CodinaPrincipeBaiges2009} ($\epsilon=+1$).
\end{proposition}
\begin{proof}
Let $\du,\dv\in V_{h,0}$. Each term of $\eta_{h}$ in \eqref{bilinear_NCIP} is a boundary integral over $\Fhb$ of a quantity involving $\du$ or $\dv$; since these vanish on $\bnd{\domain}$, we have $\eta_{h}\equiv 0$. Hence $a^{\mathfrak c}_h=a+\epsilon\,s_h$, the interior penalty $s_h$ being unchanged, which is \eqref{CCIP_bilinear}. As $\eta_{h}$ is the only non-symmetric contribution to $a_h$ and both $a$ and $s_{h}$ are symmetric, $a^{\mathfrak c}_h$ is symmetric for every $\epsilon$.
\end{proof}
The solutions of \eqref{NCIP_formulation} and \eqref{CCIP_formulation} do not coincide but are expected to be close. This conforming variant enjoys more favorable stability and convergence properties than its nonconforming counterpart, analyzed in Section~\ref{S:4} (Remarks~\ref{rmk:coercivity-conforming} and~\ref{rmk:L2-conforming}).
\begin{remark}[Non-homogeneous Dirichlet data]
\label{rmk:inhomogeneous}
For clarity, the method and its analysis are presented for homogeneous Dirichlet conditions. The extension to non-homogeneous data $u=g$ on $\bnddom$ is standard and affects only the right-hand side: the linear form of the nonconforming method~\eqref{NCIP_formulation} becomes
\begin{equation*}
\dlin{l}{\dv}=\inerprod{f}{\dv}{0}{\Th}-\epsilon\dualprod{\bkap\grdh\dv\cdot\vect{n}}{g}{0}{\Fhb}+\dualprod{\tau g}{\dv}{0}{\Fhb},
\end{equation*}
The bilinear form is unchanged. With the exact boundary datum $g$, consistency, Galerkin orthogonality and the approximation argument extend under the same regularity assumptions. This weak enforcement is used for the Kellogg test in Section~\ref{S:5.2}. For the strongly imposed boundary variant, prescribing a discrete trace $g_h\neq g$ introduces an additional boundary lifting into the error estimate.
\end{remark}
\begin{remark}[Single-valued numerical normal flux and global conservation]
\label{rmk:conservation}
The closed-form elimination of the skeletal unknown yields a single-valued numerical normal flux $\dtsg\cdot\vect{n}$ across interior faces. Taking the constant test function $\dv\equiv1\in\Vh$ in~\eqref{NCIP_formulation} gives the global conservation identity
\begin{equation}
\label{global_conservation}
\sumT\Big(\inerprod{f}{1}{0}{\e{}}-\dualprod{\dtsg\cdot\vect{n}}{1}{0}{\bnd{\e{}}}\Big)=0,
\end{equation}
Thus, for every $\epsilon\in\{0,\pm1\}$, the total outward numerical flux through the boundary equals the integrated source term, the interior contributions cancelling pairwise by single-valuedness. This global balance applies to the formulation with weakly imposed Dirichlet conditions. It does not follow by the same argument for the strongly imposed boundary variant~\eqref{CCIP_formulation}, whose test space $V_{h,0}$ excludes constants.
\end{remark}

\section{Stability and a priori error analysis}
\label{S:4}

We now turn to the analysis of the IPN method~\eqref{NCIP_formulation}, the actual object of this work. After specifying the stabilization parameter $\tau$, we establish the consistency, discrete coercivity and boundedness of the bilinear form $a_h$, from which the well-posedness of~\eqref{NCIP_formulation} and \textit{a priori} error estimates of $\du$ in the energy- and $\LL{2}$-norms follow.

\subsection{Stability analysis}

For every integer $m\ge0$, define the discrete trace constant $C_{\mathrm{tr}}(m)$ by
\begin{equation}
\label{discret_trace_ineq}
\norm{p}{0}{\f{}}^2\le C_{\mathrm{tr}}^2(m)
\frac{\abs{\f{}}}{\abs{\e{}}}\norm{p}{0}{\e{}}^2,
\qquad p\in\Pk{m}{\e{}},\quad \f{}\in\FE.
\end{equation}
The constant is chosen uniformly over the mesh, and the inequality holds componentwise for vector-valued polynomials. On a $d$-simplex, the sharp bound of Warburton and Hesthaven~\cite{WarburtonHesthaven2003} gives
\begin{equation}
\label{eq:trace-degree}
C_{\mathrm{tr}}^2(m)=\frac{(m+1)(m+d)}{d}.
\end{equation}
The geometric dependence is carried explicitly by the local ratio $\abs{\f{}}/\abs{\e{}}$.

We now specify the stabilization function $\tau$ left abstract in~\eqref{tracesg}. We use the degree-dependent normalization
\begin{equation}
\label{CT_value}
\C{T}\eqbydef C_{\mathrm{tr}}(k-1),\qquad
\C{T}^2=\frac{k(k+d-1)}{d},
\end{equation}
and, following the transmissibility-based construction of~\cite{EtangsaleUIPDG2026}, set
\begin{equation}
\label{penalty_term}
\tau\vert_{\e{},\f{}}\eqbydef\alpha_{0}\,\C{T}^2\,\kappa_{\e{},\f{}}\,\dfrac{\abs{\f{}}}{\abs{\e{}}},\qquad\kappa_{\e{},\f{}}\eqbydef\vect{n}^{T}_{\e{},\f{}}\bkap_{\e{}}\vect{n}_{\e{},\f{}},
\end{equation}
where $\alpha_0>0$ and $\kappa_{\e{},\f{}}\abs{\f{}}/\abs{\e{}}$ is the local transmissibility coefficient. The normalization $\C{T}$ is the trace constant at degree $k-1$, the degree of the discrete gradients entering the flux estimates below. Being defined elementwise, $\tau$ is positive and double-valued on $\Fhi$, and the weights $(\vomega,\varrho)$ of Lemma~\ref{lemma_Trace} are determined accordingly. Diffusion-dependent penalty weights and conjugate weighted averages of this kind originate in the weighted interior penalty of Ern, Stephansen and Zunino~\cite{ESZIMA08,BurmanZunino2006}.

The exact solution is assumed to belong to $V_s$, where
\begin{equation*}
V_{s}\eqbydef\Big\{v\in\hilbert{1}{\domain}{0}:\ v\vert_{\e{}}\in\hilbert{1+s}{\e{}}{}\ \forall\e{}\in\Th,\ \text{and}\ \bkap\grd v\in\boldsymbol{H}(\mathrm{div},\domain)\Big\},\qquad\ohalf< s\leq k,
\end{equation*}
the membership $\bkap\grd v\in\boldsymbol{H}(\mathrm{div},\domain)$ expressing that the normal diffusive flux is single-valued across the material interfaces. For the continuity estimates, we use the larger space
\begin{equation*}
V(h)\eqbydef\bigcup_{1/2<t\leq k}\big(\hilbert{1}{\domain}{}\cap\hilbert{1+t}{\Th}{}\big).
\end{equation*}
It contains $V_s$, $\Vh$ and their sum, and allows the primal and dual solutions to have different regularities. Its elementwise gradients have $\LL{2}$ face traces, since $t>1/2$. For all $v\in V(h)$, set
\begin{equation}
    \label{nrj_norm1}
    \tnorm{v}^2_{h}\eqbydef\norm{\bkap^{\ohalf}\grdh{v}}{0}{\Th}^2+\hdgseminorm{v}{\partial}^2,
\end{equation}
where $\hdgseminorm{\cdot}{\partial}$ corresponds to the mesh-dependent seminorm of $v$ on $\bnd{\domain}$ which is also diffusion-dependent and is given by:
\begin{equation}
    \hdgseminorm{v}{\partial}^2\eqbydef\sumFb\hdgseminorm{v}{\partial,\f{}}^2\quad\textrm{where}\quad\hdgseminorm{v}{\partial,\f{}}\eqbydef\norm{\tau^{\ohalf}v}{0}{\f{}}.
\end{equation}
In the sequel, $\C{}>0$ denotes a generic constant that is independent of $h$, $h_{\e{}}$ and $\bkap$, unless explicitly stated otherwise. In particular, the constants in the duality estimates below may depend on the diffusion coefficients through the piecewise elliptic regularity constant. We will also use the following notations $a\lsim b$ (resp. $a\gsim b$) extensively, meaning that there exists a generic constant $\C{}>0$ such that $a\leq\C{}b$ (resp. $a\geq\C{}b$). We introduce the stabilization-weighted normal-flux seminorm
\begin{equation}\label{aug_seminorm}
\fluxseminorm{v}^2\eqbydef\sumT\sum_{\f{}\in\FE}
\norm{\tau_{\e{},\f{}}^{-\ohalf}(\bkap\grdh v)|_{\e{}}\cdot\vect n_{\e{},\f{}}}{0}{\f{}}^2.
\end{equation}
It is finite for $v\in V(h)$ and is unchanged by the minus sign in the conservative diffusive flux.
\begin{lemma}[Extended interface bounds]\label{lemma_upper-bounds}
For all $v,w\in V(h)$,
\begin{subequations}
\begin{align}
\abs{\dualprod{\bkap\grdh v\cdot\vect n}{w}{0}{\Fhb}}
&\le\fluxseminorm{v}\hdgseminorm{w}{\partial},\label{constency_bound}\\
\norm{\varrho^{\ohalf}\jump{\bkap\grdh v}}{0}{\Fhi}
&\le\fluxseminorm{v}.\label{interior_penalty_bound}
\end{align}
\end{subequations}
\end{lemma}
\begin{proof}
The boundary inequality is weighted Cauchy--Schwarz.
On an interior face, set $q_i=(\bkap\grdh v)|_{\e{i}}\cdot\vect n_i$ and $\tau_i=\tau_{\e{i},\f{}}$. Since $\varrho=(\tau_1+\tau_2)^{-1}$,
\[
\varrho(q_1+q_2)^2\le\tau_1^{-1}q_1^2+\tau_2^{-1}q_2^2.
\]
Integrating and summing over interior faces proves the second inequality. These estimates require only positive stabilizations.
\end{proof}
The constitutive normal-flux estimate and~\eqref{penalty_term} give, for each element--face pair,
\begin{equation}\label{weighted_flux_bound}
\norm{\tau_{\e{},\f{}}^{-\ohalf}(\bkap\grdh v)|_{\e{}}\cdot\vect n_{\e{},\f{}}}{0}{\f{}}^2
\le\frac{1}{\alpha_0\C{T}^2}\frac{|\e{}|}{|\f{}|}\norm{\bkap^{\ohalf}\grdh v}{0}{\f{}}^2.
\end{equation}
For $\dv\in\Vh$, elementwise constancy of $\bkap$ makes $\bkap^{\ohalf}\grdh\dv$ a polynomial of degree at most $k-1$. Applying~\eqref{discret_trace_ineq} at this degree and summing over faces yields
\begin{equation}\label{discrete-trace-control}
\fluxseminorm{\dv}\le\C{0}\norm{\bkap^{\ohalf}\grdh\dv}{0}{\Th},\qquad
\C{0}\eqbydef(\eta_0/\alpha_0)^{\ohalf}.
\end{equation}
\begin{corollary}[Discrete interface bounds]
\label{cor:discrete-bounds}
Let $\C{0}=(\eta_0/\alpha_0)^{\ohalf}$ be as in~\eqref{discrete-trace-control}. For all $\dv\in\Vh$ and all $w\in V(h)$,
\begin{subequations}
\begin{align}
\label{disc_constency_bound}
\abs{\dualprod{\bkap\grdh\dv\cdot\vect{n}}{w}{0}{\Fhb}}&\leq\C{0}\,\norm{\bkap^{\ohalf}\grdh\dv}{0}{\Th}\,\hdgseminorm{w}{\partial},\\
\label{disc_interior_penalty_bound}
\norm{\varrho^{\ohalf}\jump{\bkap\grdh\dv}}{0}{\Fhi}&\leq\C{0}\,\norm{\bkap^{\ohalf}\grdh\dv}{0}{\Th}.
\end{align}
\end{subequations}
These follow from Lemma~\ref{lemma_upper-bounds} with $v\eqbydef\dv$ and the discrete-trace control~\eqref{discrete-trace-control}.
\end{corollary}
\begin{lemma}[Discrete coercivity and well-posedness]
\label{lemma:coercivity}
Assume $\alpha_0>2\eta_0$ (equivalently, $\C{0}^2<\tfrac12$ in Corollary~\ref{cor:discrete-bounds}). Then, for every $\epsilon\in\{0,\pm1\}$,
\begin{equation}
\label{coercivity}
\dbilin{a_h}{\dv}{\dv}\geq \C{sta}\,\tnorm{\dv}^2_h,\qquad\forall\dv\in\Vh,
\end{equation}
where
\begin{equation}
\label{coercivity_constants}
\C{sta}\eqbydef\tfrac12\Big(2-\C{0}^{2}-\C{0}\sqrt{\C{0}^{2}+4}\Big)>0.
\end{equation}
The constant $\C{sta}$ is common to the three variants, depends only on $\alpha_{0}$ and $\eta_{0}$ through $\C{0}$, and is independent of $h$, of the polynomial degree and of the diffusion tensor.
Consequently, the discrete problem \eqref{NCIP_formulation} admits a unique solution, which moreover satisfies the stability bound
\begin{equation}
\label{eq:stabbound}
\tnorm{u_h}_h \lsim \C{sta}^{-1}\underline{\kappa}_{0}^{-\ohalf}\,\Vert f \Vert_{0,\Omega}.
\end{equation}
\end{lemma}
\begin{proof}
Let $\dv\in\Vh$ and set $X\eqbydef\norm{\bkap^{\ohalf}\grdh\dv}{0}{\Th}$ and $Y\eqbydef\hdgseminorm{\dv}{\partial}$. Using the definitions of $s_h$ and $\eta_h$, with $\norm{\tau^{\ohalf}\dv}{0}{\Fhb}^2=Y^2$, gives
\begin{equation*}
\dbilin{a_h}{\dv}{\dv}=X^2-\epsilon\,\norm{\varrho^{\ohalf}\jump{\bkap\grdh\dv}}{0}{\Fhi}^2-(1+\epsilon)\dualprod{\bkap\grdh\dv\cdot\vect{n}}{\dv}{0}{\Fhb}+Y^2.
\end{equation*}
If $\epsilon=-1$, the second term is nonnegative and the third vanishes, so $\dbilin{a_h}{\dv}{\dv}\geq X^2+Y^2=\tnorm{\dv}_h^2$ for any $\alpha_0>0$. If $\epsilon\in\{0,+1\}$, the discrete interface bounds of Corollary~\ref{cor:discrete-bounds} yield
\begin{equation*}
\dbilin{a_h}{\dv}{\dv}\geq\big(1-\C{0}^2\abs{\epsilon}\big)X^2-(1+\epsilon)\,\C{0}\,XY+Y^2.
\end{equation*}
Since $X,Y\geq0$, all three variants are bounded below by $(1-\C{0}^2)X^2-2\C{0}XY+Y^2$. The smallest eigenvalue of this quadratic form is $\C{sta}$ in~\eqref{coercivity_constants}, which is positive if and only if $\C{0}^2<\tfrac12$, that is, $\alpha_0>2\eta_0$. This proves~\eqref{coercivity}.

Existence and uniqueness follow since \eqref{NCIP_formulation} is a coercive, linear, finite-dimensional problem. Finally, we prove the stability bound \eqref{eq:stabbound}. Since the Dirichlet condition is only weakly enforced, $\Vh\not\subset\hilbert{1}{\domain}{0}$, and the Poincar\'e inequality on $\hilbert{1}{\domain}{0}$ does not apply to $\du$. We invoke instead the generalized Poincar\'e--Friedrichs inequality on $\hilbert{1}{\domain}{}$ (see, \eg~\cite{ErnGuermond2004}): there exists $\C{PF}>0$, depending only on $\domain$, such that $\norm{v}{0}{\domain}\leq\C{PF}\big(\norm{\grd v}{0}{\domain}+\norm{v}{0}{\bnddom}\big)$ for all $v\in\hilbert{1}{\domain}{}$, and control each right-hand side term by the energy norm \eqref{nrj_norm1}. The gradient term satisfies $\norm{\grd\du}{0}{\domain}\leq\underline{\kappa}_{0}^{-\ohalf}\norm{\bkap^{\ohalf}\grdh\du}{0}{\Th}$. For the boundary term, on each boundary face $\f{}\subset\bnd{\e{}}\cap\bnddom$, the penalty in the form~\eqref{penalty_term} gives $\norm{\du}{0}{\f{}}=\tau^{-\ohalf}\hdgseminorm{\du}{\partial,\f{}}=(\alpha_0\C{T}^2)^{-\ohalf}\,\big(\abs{\e{}}/(\abs{\f{}}\kappa_{\e{},\f{}})\big)^{\ohalf}\hdgseminorm{\du}{\partial,\f{}}$; since $\kappa_{\e{},\f{}}\geq\underline{\kappa}_{\e{}}\geq\underline{\kappa}_{0}$ and $\abs{\e{}}/\abs{\f{}}\lsim h_{\e{}}\leq h\leq1$ by shape-regularity, squaring and summing over $\Fhb$ yields $\norm{\du}{0}{\bnddom}\leq(\alpha_0\C{T}^2)^{-\ohalf}\,\underline{\kappa}_{0}^{-\ohalf}\hdgseminorm{\du}{\partial}$. Collecting these two bounds, we infer that $\norm{\du}{0}{\domain}\leq\sqrt2\,\C{PF}\max(1,(\alpha_0\C{T}^2)^{-\ohalf})\,\underline{\kappa}_{0}^{-\ohalf}\tnorm{\du}_{h}$, whence $\dlin{l}{\du}=(f,\du)_{0,\domain}\lsim\underline{\kappa}_{0}^{-\ohalf}\norm{f}{0}{\domain}\tnorm{\du}_{h}$, the hidden constant depending on $\domain$ and, through $(\alpha_0\C{T}^2)^{-\ohalf}$, on the penalty parameter $\alpha_0$, but neither on $h$ nor on $\bkap$. Combining this bound with the coercivity \eqref{coercivity} yields \eqref{eq:stabbound}.
\end{proof}
\begin{remark}[Per-variant penalty thresholds]
\label{rmk:thresholds}
The same proof gives the sharper lower-bound coefficients $1$ for $\epsilon=-1$ and $1-\C{0}/2$ for $\epsilon=0$. Thus the sufficient thresholds on the penalty parameter are
\begin{equation}
\label{thresholds}
\epsilon=-1:\ \alpha_0>0,\qquad \epsilon=0:\ \alpha_0>\tfrac{\eta_0}{4},\qquad \epsilon=+1:\ \alpha_0>2\eta_0,
\end{equation}
These thresholds recover unconditional coercivity for N-IPN, as in non-symmetric Nitsche methods~\cite{FreundStenberg1995,Burman2012}; on simplices, the S-IPN condition is $\alpha_0>2(d+1)$. Here ``unconditional'' means that no positive stability threshold is required: the boundary control in $\hdgseminorm{\cdot}{\partial}$ still weakens as $\alpha_0\to0$.
\end{remark}
\begin{remark}[Relaxed coercivity in the conforming setting]
\label{rmk:coercivity-conforming}
For the $\HH{1}_{0}$-conforming variant~\eqref{CCIP_formulation} of Proposition~\ref{prop:conforming}, the boundary mechanism $\eta_{h}$ automatically vanishes and the natural energy-norm reduces to $\tnorm{\dv}_{h}=\norm{\bkap^{\ohalf}\grdh\dv}{0}{\Th}$ on $V_{h,0}$. The coercivity conditions of Lemma~\ref{lemma:coercivity} then relax to
\begin{equation*}
\dbilin{a^{\mathfrak c}_h}{\dv}{\dv}\geq\begin{cases}
\norm{\bkap^{\ohalf}\grdh\dv}{0}{\Th}^2 & \textrm{if }\epsilon\in\{0,-1\}\ \textrm{(unconditionally)},\\[2pt]
(1-\eta_0/\alpha_0)\,\norm{\bkap^{\ohalf}\grdh\dv}{0}{\Th}^2 & \textrm{if }\epsilon=+1,\ \textrm{provided }\alpha_0>\eta_0,
\end{cases}
\end{equation*}
the anti-penalized CIP threshold $\alpha_0>\eta_0$ ($\epsilon=+1$) being a genuine relaxation, by a factor two, of the IPN condition $\alpha_0>2\eta_0$ of Lemma~\ref{lemma:coercivity}.
\end{remark}

\subsection{Convergence analysis of the native discrete variable}

We now assume the natural piecewise regularity for the exact solution of \eqref{weak_form}, \ie $u\in V_{s}$. We further equip $V(h)$ with the enriched energy-norm defined, for all $v\in V(h)$, by
\begin{equation}
    \label{nrj_norm2}
    \tnorm{v}_{h,\ast}^2\eqbydef\tnorm{v}^2_{h}+\fluxseminorm{v}^2,
\end{equation}
where $\fluxseminorm{\cdot}$ is the augmented seminorm of the diffusive-flux traces introduced in~\eqref{aug_seminorm}. Let us emphasize that both respective norms $\tnorm{\cdot}_{h}$ and $\tnorm{\cdot}_{h,\ast}$ are uniformly equivalent on $\Vh$; \ie for all $\dv\in\Vh$ then one may take $\C{\ast}=(1+\C{0}^2)^{-\ohalf}$ in
\begin{equation}
    \label{norm_equivalence}
    \C{\ast}\tnorm{\dv}_{h,\ast}\leq\tnorm{\dv}_{h}\leq\tnorm{\dv}_{h,\ast}.
\end{equation}
The upper bound of \eqref{norm_equivalence} is immediate, while its lower bound follows from~\eqref{discrete-trace-control}. %

\begin{lemma}[Strong consistency] 
\label{lemma:consistency}
Let $u\in V_{s}$ be the exact solution of \eqref{weak_form}. For all $\dv\in\Vh{}$, the following holds:
\begin{equation}
    \label{consistency}
    \dbilin{a_{h}}{u}{\dv}=\inerprod{f}{\dv}{0}{\Th}.%
\end{equation}
\end{lemma}
\begin{proof}
Since $u\in\hilbert{1}{\domain}{0}$, its trace vanishes on $\bnd{\domain}$, so that $\dbilin{\eta_{h}}{u}{\dv}=\dualprod{-\bkap\grdh u\cdot\vect{n}}{\dv}{0}{\Fhb}$. Moreover, since $\bkap\grd u\in\boldsymbol{H}(\mathrm{div},\domain)$, the normal diffusive flux is single-valued across interior interfaces, \ie $\jump{\bkap\grd u}=0$ on $\Fhi$. We then deduce that (i) the IP contribution is null, \ie $\dbilin{s_{h}}{u}{\dv}=0$ and (ii) the Nitsche term can be rewritten as follows, \ie $\dbilin{\eta_{h}}{u}{\dv}=-\dualprod{\bkap\grd u\cdot\vect{n}}{\dv}{0}{\bnd{\Th}}.$
Combining these results, we finally infer that 
\begin{align*}
    \dbilin{a_h}{u}{\dv}&=\sumT\bigg[\inerprod{\bkap\grd u}{\grdh\dv}{0}{\e{}}-\dualprod{\bkap\grd u\cdot\vect{n}}{\dv}{0}{\bnd{\e{}}}\bigg],\\
    &=\sumT\inerprod{\grd\cdot(-\bkap\grdh u)}{\dv}{0}{\e{}}=\sumT\inerprod{f}{\dv}{0}{\e{}},
\end{align*}
which is the statement of the lemma.
\end{proof}

\begin{remark}[Exact consistency of the interior mechanism]
\label{rmk:exact-consistency}
The normal continuity of the exact diffusive flux implies
\begin{equation*}
    \dbilin{s_{h}}{u}{\dv}=0,\qquad\forall\dv\in\Vh.
\end{equation*}
Hence the interior stabilization contributes no residual to the Galerkin orthogonality relation~\eqref{galerkin_ortho}.
\end{remark}
An immediate consequence of Lemma \ref{lemma:consistency} is the Galerkin orthogonality property; \ie 
\begin{equation}
    \label{galerkin_ortho}
    \dbilin{a_{h}}{u-\du}{\dv}=0,\quad\forall\dv\in\Vh{}.
\end{equation}
\begin{lemma}[Discrete boundedness] 
\label{lemma:boundedness}
For all $(v,\dw)\in V(h)\times\Vh$,
    \begin{equation}
        \label{boundedness}
        \abs{\dbilin{a_h}{v}{\dw}}\leq\C{b}\,\tnorm{v}_{h,\ast}\tnorm{\dw}_{h},
        \qquad\C{b}\eqbydef2(1+\abs{\epsilon}\C{0}).
    \end{equation}
The constant $\C{b}$ is independent of $h$ and of the diffusion tensor, and satisfies
$\C{b}\leq2(1+\C{0})$ uniformly in $\epsilon\in\{0,\pm1\}$.
\end{lemma}
\begin{proof}
The volume and boundary-penalty terms are bounded together by
$\tnorm{v}_h\tnorm{\dw}_h$ using the Cauchy--Schwarz inequality.
For the remaining terms, apply~\eqref{constency_bound} to $(v,\dw)$,
\eqref{disc_constency_bound} to $(\dw,v)$, and use~\eqref{interior_penalty_bound}
and~\eqref{disc_interior_penalty_bound} to bound the two factors in $s_h$. This gives
\begin{equation*}
\begin{aligned}
\abs{\dbilin{a_h}{v}{\dw}}
&\leq\tnorm{v}_h\tnorm{\dw}_h
+\fluxseminorm{v}\hdgseminorm{\dw}{\partial}+\abs{\epsilon}\C{0}
\Big(\hdgseminorm{v}{\partial}+\fluxseminorm{v}\Big)
\norm{\bkap^{\ohalf}\grdh\dw}{0}{\Th}\\
&\leq\C{b}\,\tnorm{v}_{h,\ast}\tnorm{\dw}_h,
\end{aligned}
\end{equation*}
which proves~\eqref{boundedness}.
\end{proof}
\begin{remark}[Extended boundedness]
\label{extended_boundedness}
For all $v,w\in V(h)$,
\begin{equation*}
\abs{\dbilin{a_h}{v}{w}}\leq\C{b}^{\ast}\,\tnorm{v}_{h,\ast}\tnorm{w}_{h,\ast},
\qquad\C{b}^{\ast}\eqbydef2(1+\abs{\epsilon}).
\end{equation*}
The proof follows that of Lemma~\ref{lemma:boundedness}, using the extended interface
bounds of Lemma~\ref{lemma_upper-bounds} for both arguments.
This extension is used in the duality argument of Theorem~\ref{thm:L2-estimates}.
\end{remark}
We now derive (i) an energy error estimate, through the standard coercivity--boundedness argument for strongly consistent nonconforming methods (see, \eg~\cite{di2011mathematical}), and (ii) a sharper $\LL{2}$-error estimate by an Aubin--Nitsche duality argument.
\begin{lemma}[Local approximation in the enriched energy norm]
\label{lemma:local-interpolation}
Let $d\in\{2,3\}$, $\ohalf<\mu\leq k$, and let $v\in\hilbert{1}{\domain}{}$ satisfy $v\vert_{\e{}}\in\hilbert{1+\mu}{\e{}}{}$ for every $\e{}\in\Th$. The degree-$k$ nodal Lagrange interpolant $I_hv\in\Vh$ is well defined, preserves a homogeneous boundary trace, and satisfies
\begin{equation}
\label{local-interpolation}
\tnorm{v-I_hv}_{h,\ast}^{2}\lsim
\sumT\overline{\kappa}_{\e{}}h_{\e{}}^{2\mu}\seminorm{v}{1+\mu}{\e{}}^{2}.
\end{equation}
The constant depends on $k$, $\mu$, $\alpha_0$ and mesh shape-regularity, but not on the diffusion tensors or on $h$.
\end{lemma}
\begin{proof}
Since $1+\mu>d/2$, each element restriction has a continuous representative. Their traces agree on shared faces because $v\in H^1(\domain)$; continuity then makes the shared nodal values agree. Thus the local Lagrange interpolants assemble into $I_hv\in\Vh$, with zero boundary nodal values whenever $v\in H^1_0(\domain)$.

Write $e_I=v-I_hv$. On the reference simplex, point evaluation is bounded on $H^{1+\mu}$, and the gradient trace is bounded from $H^{1+\mu}$ to $L^2$ of the boundary because $\mu>1/2$. Polynomial reproduction, the polynomial approximation estimate on the reference simplex, and affine scaling therefore give the local bounds (see also~\cite{ErnGuermond2004})
\begin{align*}
\seminorm{e_I}{1}{\e{}}&\lsim h_{\e{}}^\mu\seminorm{v}{1+\mu}{\e{}},\\
\norm{e_I}{0}{\bnd{\e{}}}&\lsim h_{\e{}}^{\mu+1/2}\seminorm{v}{1+\mu}{\e{}},\\
\norm{\grd e_I}{0}{\bnd{\e{}}}&\lsim h_{\e{}}^{\mu-1/2}\seminorm{v}{1+\mu}{\e{}}.
\end{align*}
These statements also hold for fractional $\mu$; the constant is not asserted to remain bounded as $\mu\downarrow1/2$. The elementwise inequality $\bkap_{\e{}}\preceq\overline{\kappa}_{\e{}}\mathrm I$ bounds the volume contribution to $\tnorm{e_I}_{h,\ast}^{2}$ by the right-hand side of~\eqref{local-interpolation}. Shape-regularity and~\eqref{penalty_term} give $\tau\vert_{\e{},\f{}}\lsim\overline{\kappa}_{\e{}}h_{\e{}}^{-1}$, so the second bound controls the boundary penalty contribution. Finally,~\eqref{weighted_flux_bound} gives
\begin{equation*}
\fluxseminorm{e_I}^2\le\frac{1}{\alpha_0\C{T}^2}\sumT\sum_{\f{}\in\FE}\frac{\abs{\e{}}}{\abs{\f{}}}
\norm{\bkap^{\ohalf}\grd e_I}{0}{\f{}}^2
\lsim\sumT\overline{\kappa}_{\e{}}h_{\e{}}
\norm{\grd e_I}{0}{\bnd{\e{}}}^{2}
\lsim\sumT\overline{\kappa}_{\e{}}h_{\e{}}^{2\mu}\seminorm{v}{1+\mu}{\e{}}^{2}.
\end{equation*}
This controls the remaining contribution and proves the assertion. All estimates are elementwise; no averaging across material interfaces is required.
\end{proof}
\begin{theorem}[$\tnorm{\cdot}_{h,\ast}$-norm error estimate]\label{thm:nrj-estimate}
Let $u\in V_s$ solve~\eqref{weak_form} and $\du\in\Vh$ solve~\eqref{NCIP_formulation}. Under the hypotheses of Lemma~\ref{lemma:coercivity}, with $\mu\eqbydef\min(k,s)$,
\begin{equation}
\label{nrj-estimate}
\tnorm{u-\du}_{h,\ast}\lsim
\bigg(\sumT\overline{\kappa}_{\e{}}h_{\e{}}^{2\mu}\seminorm{u}{1+\mu}{\e{}}^2\bigg)^{\ohalf}
\lsim\overline{\kappa}_0^{\ohalf}h^\mu\seminorm{u}{1+\mu}{\Th}.
\end{equation}
The hidden constant is independent of $h$, of the coefficient jumps and of the local anisotropy ratios. The displayed weights and the broken regularity seminorm retain their dependence on the diffusion problem. No quasi-monotonicity assumption on the coefficients is needed.
\end{theorem}
\begin{proof}
For any $\dv\in\Vh$, discrete coercivity, Galerkin orthogonality~\eqref{galerkin_ortho} and boundedness~\eqref{boundedness} yield
\begin{equation*}
\C{sta}\tnorm{\du-\dv}_h^2
\leq\dbilin{a_h}{\du-\dv}{\du-\dv}
=\dbilin{a_h}{u-\dv}{\du-\dv}
\lsim\tnorm{u-\dv}_{h,\ast}\tnorm{\du-\dv}_h.
\end{equation*}
The discrete norm equivalence~\eqref{norm_equivalence} and the triangle inequality imply
$\tnorm{u-\du}_{h,\ast}\lsim\inf_{\dv\in\Vh}\tnorm{u-\dv}_{h,\ast}$.
Choose $\dv=I_hu$ and apply Lemma~\ref{lemma:local-interpolation}.
\end{proof}
To sharpen this result in the weaker $\LL{2}$-norm, we resort to an Aubin--Nitsche duality argument, which requires a regularity assumption on the dual solution. We assume that $\domain$ admits a \emph{fixed} partition $\mathcal{P}\eqbydef\{\domain_i\}_{1\leq i\leq N}$ into finitely many polyhedral subdomains, resolved by every mesh of the family $\{\Th\}_{h>0}$ and such that the discontinuities of $\bkap$ are aligned with the subdomain boundaries. Let $\er{u}\eqbydef u-\du \in V(h)$ and let $\psi$ solve the weak adjoint problem
\begin{equation}
    \label{weak_adj_prob}
    \dbilin{a}{v}{\psi}=\inerprod{\er{u}}{v}{0}{\domain},\quad\forall v\in\hilbert{1}{\domain}{0}.
\end{equation}
We then assume the \emph{piecewise elliptic regularity property}: $\psi\vert_{\domain_i}\in\hilbert{2}{\domain_i}{}$ for all $\domain_i\in\mathcal{P}$, $\bkap\grd\psi\in\boldsymbol{H}(\mathrm{div},\domain)$ (so that the dual normal flux is single-valued across the interfaces), and there exists $\C{ell}>0$, depending on $\domain$, $\mathcal{P}$ and the distribution of $\bkap$, such that
\begin{equation}
    \label{elliptic_regularity}
    \norm{\psi}{2}{\mathcal{P}}\eqbydef\Big(\norm{\psi}{1}{\domain}^2+\textstyle\sum_{i=1}^{N}\seminorm{\psi}{2}{\domain_i}^2\Big)^{\ohalf}\leq\C{ell}\norm{\er{u}}{0}{\domain}.
\end{equation}
Since the meshes resolve $\mathcal{P}$, this implies $\psi\in\hilbert{2}{\Th}{}$ with $\seminorm{\psi}{2}{\Th}^2=\sum_{i=1}^{N}\seminorm{\psi}{2}{\domain_i}^2\leq\norm{\psi}{2}{\mathcal{P}}^2$. Since the strong form of~\eqref{weak_adj_prob} reads $-\grd\cdot(\bkap\grd\psi)=\er{u}$ in $\domain$ with $\bkap\grd\psi\in\boldsymbol{H}(\mathrm{div},\domain)$, elementwise integration by parts, together with the single-valuedness of $\er{u}$ (as $\er{u}\in\hilbert{1}{\domain}{}$) and of $\bkap\grd\psi\cdot\vect{n}$ across $\Fhi$, gives
\begin{equation*}
\norm{\er{u}}{0}{\domain}^{2}=\dbilin{a}{\er{u}}{\psi}-\dualprod{\bkap\grd\psi\cdot\vect{n}}{\er{u}}{0}{\Fhb}.
\end{equation*}
On the other hand, $\psi=0$ on $\bnd{\domain}$ and $\jump{\bkap\grd\psi}=0$ on $\Fhi$ make the interior-penalty term $\dbilin{s_h}{\er{u}}{\psi}$ and the $\psi$-bearing boundary contributions of $a_h$ vanish, leaving $\dbilin{a_h}{\er{u}}{\psi}=\dbilin{a}{\er{u}}{\psi}-\epsilon\dualprod{\bkap\grd\psi\cdot\vect{n}}{\er{u}}{0}{\Fhb}$. Eliminating $\dbilin{a}{\er{u}}{\psi}$ between these two identities yields
\begin{align}
    \label{L2-error}
    \norm{\er{u}}{0}{\domain}^{2}&=\dbilin{a_{h}}{\er{u}}{\psi}+(\epsilon-1)\dualprod{\bkap\grd\psi\cdot\vect{n}}{\er{u}}{0}{\Fhb}.%
\end{align}
It then remains to bound the two contributions on the right-hand side of \eqref{L2-error}, which leads to the following theorem.
\begin{theorem}[$\LL{2}$-error estimate] 
\label{thm:L2-estimates}
Let $u\in V_{s}$ solve \eqref{weak_form} and $\du\in\Vh$ solve \eqref{NCIP_formulation}. Besides the hypotheses of Lemma~\ref{lemma:coercivity}, assume the piecewise elliptic regularity property~\eqref{elliptic_regularity}. Then, with a hidden constant depending on the piecewise elliptic-regularity constant $\C{ell}$ and independent of $h$; since $\C{ell}$ is in general contrast-dependent, no uniformity with respect to $\bkap$ is claimed (Remark~\ref{rmk:dual-regularity}),
\begin{equation}
    \label{L2-estimate_optimal}
    \norm{u-\du}{0}{\domain}\lsim\overline{\kappa}_{0}^{\ohalf}\big(h+\abs{\epsilon-1}\,h^{\ohalf}\big)\tnorm{u-\du}_{h,\ast},
\end{equation}
the two contributions being compared under the normalization $h\leq1$ of Section~\ref{S:2}. Combining \eqref{L2-estimate_optimal} with \eqref{nrj-estimate} yields
\begin{equation}
    \label{L2-estimate_rate}
    \norm{u-\du}{0}{\domain}\lsim\overline{\kappa}_{0}\big(h^{\mu+1}+\abs{\epsilon-1}\,h^{\mu+\ohalf}\big)\seminorm{u}{\mu+1}{\Th}.
\end{equation}
The rate is optimal if $\epsilon=1$ and admits at most a half-order loss otherwise.
\end{theorem}
\begin{proof}
 Let us bound the first term on the right-hand side of \eqref{L2-error}. Set $\proj{\psi}\eqbydef I_h\psi$, using Lemma~\ref{lemma:local-interpolation} with $\mu=1$. Observing that $\proj{\psi}\in\Vh$, we infer that $\dbilin{a_{h}}{\er{u}}{\proj{\psi}}=0$ from the Galerkin orthogonality \eqref{galerkin_ortho} and hence $\dbilin{a_{h}}{\er{u}}{\psi}=\dbilin{a_{h}}{\er{u}}{\erp{\psi}{\pi}}$ where $\erp{\psi}{\pi}\eqbydef\psi-\proj{\psi}$. Now, using the extended boundedness of $a_h$ on $V(h)\times V(h)$, \ie $\abs{\dbilin{a_{h}}{\er{u}}{\erp{\psi}{\pi}}}\lsim\tnorm{\er{u}}_{h,\ast}\tnorm{\erp{\psi}{\pi}}_{h,\ast}$, the local interpolation estimate~\eqref{local-interpolation} in the enriched energy norm, \ie $\tnorm{\erp{\psi}{\pi}}_{h,\ast}\lsim\overline{\kappa}_{0}^{\ohalf}h\seminorm{\psi}{2}{\Th}$ (the global factor $\overline{\kappa}_{0}^{\ohalf}$ bounding all the diffusion weights entering $\tnorm{\cdot}_{h,\ast}$), and the elliptic regularity property \eqref{elliptic_regularity}, \ie $\seminorm{\psi}{2}{\Th}\leq\norm{\psi}{2}{\mathcal{P}}\leq\C{ell}\norm{\er{u}}{0}{\domain}$, we then obtain the following upper bound,
 \begin{equation}
         \abs{\dbilin{a_{h}}{\er{u}}{\psi}}\lsim\overline{\kappa}_{0}^{\ohalf}h\tnorm{\er{u}}_{h,\ast}\norm{\er{u}}{0}{\domain}.
 \end{equation}
 For the second term, the Cauchy--Schwarz inequality with the weight $\tau^{\ohalf}$ on each boundary face, together with $\tau^{-\ohalf}=(\alpha_0\C{T}^2)^{-\ohalf}\big(\abs{\e{}}/(\abs{\f{}}\kappa_{\e{},\f{}})\big)^{\ohalf}$, $\norm{\bkap\grd\psi\cdot\vect{n}}{0}{\f{}}\leq\kappa_{\e{},\f{}}^{\ohalf}\norm{\bkap^{\ohalf}\grd\psi}{0}{\f{}}$ and $\abs{\e{}}/\abs{\f{}}\lsim h$, gives
\begin{equation*}
    \abs{\dualprod{\bkap\grdh\psi\cdot\vect{n}}{\er{u}}{0}{\Fhb}}\leq(\alpha_0\C{T}^2)^{-\ohalf}h^{\ohalf}\norm{\bkap^{\ohalf}\grd\psi}{0}{\bnd{\domain}}\hdgseminorm{\er{u}}{\partial}\leq(\alpha_0\C{T}^2)^{-\ohalf}(\overline{\kappa}_{0}h)^{\ohalf}\norm{\grd\psi}{0}{\bnd{\domain}}\hdgseminorm{\er{u}}{\partial}.
\end{equation*}
Since the partition $\mathcal{P}$ is fixed (independent of $h$) and $\psi\vert_{\domain_i}\in\hilbert{2}{\domain_i}{}$, the continuous trace theorem, applied on each subdomain $\domain_i$ abutting $\bnddom$, yields $\norm{\grd\psi}{0}{\bnd{\domain}}\leq\C{ct}\norm{\psi}{2}{\mathcal{P}}$ with $\C{ct}$ depending only on $\mathcal{P}$. Combining this with $\hdgseminorm{\er{u}}{\partial}\leq\tnorm{\er{u}}_{h,\ast}$ and the elliptic regularity \eqref{elliptic_regularity}, we obtain
\begin{equation}
    \abs{\dualprod{\bkap\grd\psi\cdot\vect{n}}{\er{u}}{0}{\Fhb}}\lsim(\overline{\kappa}_{0}h)^{\ohalf}\tnorm{\er{u}}_{h,\ast}\norm{\er{u}}{0}{\domain}.
\end{equation}
Collecting the above estimates yields \eqref{L2-estimate_optimal}, and \eqref{nrj-estimate} then gives \eqref{L2-estimate_rate}. The second term of \eqref{L2-error} vanishes for $\epsilon=1$, whence the optimality.
\end{proof}
\begin{remark}[Scope of the elliptic-regularity assumption]
\label{rmk:dual-regularity}
The piecewise elliptic regularity~\eqref{elliptic_regularity} is a genuine assumption. For a general discontinuous $\bkap$ it may fail: the dual solution then belongs only to $\hilbert{1+\gamma}{\domain_i}{}$ piecewise, with $\gamma\in(0,1)$ determined by the interface configuration and the coefficient jumps (for checkerboard-type media, $\gamma$ can be made arbitrarily small by increasing the contrast~\cite{Kellogg1975}), and the constant $\C{ell}$ is in general not uniform in the contrast. In that regime, Theorem~\ref{thm:L2-estimates} does not apply. A fractional-regularity duality argument would be needed to quantify the reduced gain; when $\gamma\leq1/2$, even the $\LL{2}$ gradient traces used here require a different treatment.
\end{remark}
\begin{remark}[Optimal $\LL{2}$-convergence in the conforming setting]
\label{rmk:L2-conforming}
The half-order loss for $\epsilon\neq 1$ is entirely caused by the boundary residual $(\epsilon-1)\dualprod{\bkap\grd\psi\cdot\vect{n}}{\er{u}}{0}{\Fhb}$ in \eqref{L2-error}, which stems from the weak (nonconforming) enforcement of the Dirichlet condition through $\eta_h$. As discussed in the conforming variant \eqref{CCIP_formulation}, this term vanishes identically when $\du$ is sought in $V_{h,0}\subset\hilbert{1}{\domain}{0}$, so that the optimal estimate $\norm{u-\du}{0}{\domain}\lsim \overline{\kappa}_{0}\,h^{\mu+1}\seminorm{u}{\mu+1}{\Th}$ holds for the three variants $\epsilon\in\{0,\pm1\}$ under the same piecewise elliptic regularity assumption~\eqref{elliptic_regularity}, and not only for the symmetric one.
\end{remark}

\begin{remark}[Coupled interior and boundary mechanisms]
\label{rmk:single-epsilon}
The closed-form trace elimination propagates the same symmetrization parameter $\epsilon$ to the interior flux-jump term and the boundary Nitsche terms. Consequently, $\epsilon=-1$ yields unconditional coercivity but loses adjoint consistency, whereas $\epsilon=1$ ensures adjoint consistency and the optimal $\LL{2}$ estimate under the stated regularity assumption, subject to the stronger penalty threshold. This coupling is inherited from the auxiliary hybrid formulation.
\end{remark}

\section{Numerical experiments}
\label{S:5}

We assess the energy estimates of Section~\ref{S:4} and explore $\LL{2}$ convergence on two benchmarks, reusing the test definitions of~\cite{EtangsaleUIPDG2026}: a heterogeneous anisotropic problem with a smooth manufactured solution (Section~\ref{S:5.1}), assessing the optimal convergence rates and their robustness with respect to the diffusion contrast; and the Kellogg problem with a singular solution (Section~\ref{S:5.2}), probing the regularity-limited regime. A third study closes the section: the conforming variant and the role of the transmissibility weighting of its interior penalty (Section~\ref{S:5.3}). Computations use the \texttt{NGSolve} library~\cite{Schoberl2014c++} on shape-regular triangular meshes conforming to the material interfaces (the numerical study is thus two-dimensional, while the analysis of Section~\ref{S:4} covers $d\in\{2,3\}$); the penalty is~\eqref{penalty_term} with~\eqref{CT_value} and, unless stated otherwise, $\alpha_0=4(d+1)$, which exceeds the sufficient S-IPN threshold $2(d+1)$ of Lemma~\ref{lemma:coercivity}. We report the three variants S-IPN ($\epsilon=1$), N-IPN ($\epsilon=-1$) and I-IPN ($\epsilon=0$).

\subsection{Test 1: heterogeneous anisotropic diffusion with a smooth solution}
\label{S:5.1}

We consider the homogeneous Dirichlet problem on $\domain=(0,1)^2$, split into four subdomains $\domain_i$, $i=1,\dots,4$, numbered counterclockwise from the bottom-left and separated by the interfaces $x=\tfrac12$ and $y=\tfrac12$ (see Figure~\ref{fig:test1-geometry}). The diffusion tensor is piecewise constant, $\bkap\vert_{\domain_i}=\bkap_i$, with
\begin{equation}
\label{test1-tensor}
\bkap_1=\bkap_3=\begin{bmatrix}\lambda&0\\0&1\end{bmatrix}\quad\textrm{and}\quad\bkap_2=\bkap_4=\begin{bmatrix}1&0\\0&\lambda^{-1}\end{bmatrix},\qquad\lambda\geq1,
\end{equation}
where $\lambda$ simultaneously controls the local anisotropy within each subdomain and the global heterogeneity of the medium; the contrast is $\kappa_{\max}/\kappa_{\min}=\lambda^2$. We prescribe the smooth manufactured solution $u=\sin(\pi x)\sin(\pi y)$, with $f=-\dvg(\bkap\grd u)$ computed per subdomain; the meshes conform to the subdomain partition, so that the discontinuities of $\bkap$ align with element interfaces. Since $\partial_x u$ vanishes on $x=\tfrac12$ and $\partial_y u$ on $y=\tfrac12$, the normal flux is continuous across both interfaces for every $\lambda$, so $u$ is the exact weak solution. The smooth primal solution satisfies the regularity required by Theorem~\ref{thm:nrj-estimate}, including at the cross-point. This does not imply piecewise $H^2$ regularity for the adjoint with a general right-hand side: condition~\eqref{elliptic_regularity} is not established for this checkerboard. The energy estimate applies, while the $\LL{2}$ convergence study is empirical.

\begin{figure}[!ht]
\centering
\begin{tikzpicture}[scale=0.5]
  \draw (0,0) rectangle (10,10);
  \draw[dashed] (0,5) -- (10,5);
  \draw[dashed] (5,0) -- (5,10);
  \draw (0,0) node[above right] {$\domain_1$};
  \draw (10,0) node[above left] {$\domain_2$};
  \draw (10,10) node[below left] {$\domain_3$};
  \draw (0,10) node[below right] {$\domain_4$};
  \draw[<->] (0.5,2.5) -- (4.5,2.5);   \draw[<->] (2.5,1.5) -- (2.5,3.5);   \draw (2.5,2.5) node[below right] {$\bkap_{1}$};
  \draw[<->] (6.5,2.5) -- (8.5,2.5);   \draw[<->] (7.5,2.1) -- (7.5,2.9);   \draw (7.5,2.5) node[below right] {$\bkap_{2}$};
  \draw[<->] (5.5,7.5) -- (9.5,7.5);   \draw[<->] (7.5,6.5) -- (7.5,8.5);   \draw (7.5,7.5) node[below right] {$\bkap_{3}$};
  \draw[<->] (1.5,7.5) -- (3.5,7.5);   \draw[<->] (2.5,7.1) -- (2.5,7.9);   \draw (2.5,7.5) node[below right] {$\bkap_{4}$};
  \draw (5,0) node[below] {(a)};
\end{tikzpicture}\hspace{1cm}
\begin{tikzpicture}[scale=0.5]
  \draw (0,0) rectangle (10,10);
  \draw[dashed] (0,5) -- (10,5);
  \draw[dashed] (5,0) -- (5,10);
  \draw (0,0) node[above right] {$\domain_1$};
  \draw (10,0) node[above left] {$\domain_2$};
  \draw (10,10) node[below left] {$\domain_3$};
  \draw (0,10) node[below right] {$\domain_4$};
  \draw[<->] (1,2.5) -- (4,2.5);   \draw[<->] (2.5,1) -- (2.5,4);   \draw (2.5,2.5) node[below right] {$\bkap_{1}$};
  \draw[<->] (7,2.5) -- (8,2.5);   \draw[<->] (7.5,2) -- (7.5,3);   \draw (7.5,2.5) node[below right] {$\bkap_{2}$};
  \draw[<->] (6,7.5) -- (9,7.5);   \draw[<->] (7.5,6) -- (7.5,9);   \draw (7.5,7.5) node[below right] {$\bkap_{3}$};
  \draw[<->] (2,7.5) -- (3,7.5);   \draw[<->] (2.5,7) -- (2.5,8);   \draw (2.5,7.5) node[below right] {$\bkap_{4}$};
  \draw (5,0) node[below] {(b)};
\end{tikzpicture}
\caption{The four-subdomain partition $\domain=\bigcup_{i=1}^{4}\domain_i$ shared by both tests; in each subdomain the tensor $\bkap_i$ is sketched through its principal directions, the length of each arrow encoding the relative magnitude of the corresponding eigenvalue (in panel (a), $\lambda$, $1$ or $\lambda^{-1}$). (a) Test~1 (Section~\ref{S:5.1}): combined anisotropy and heterogeneity with a smooth solution. (b) Test~2 (Section~\ref{S:5.2}, Kellogg): isotropic but heterogeneous coefficients with a singular solution.}
\label{fig:test1-geometry}
\end{figure}
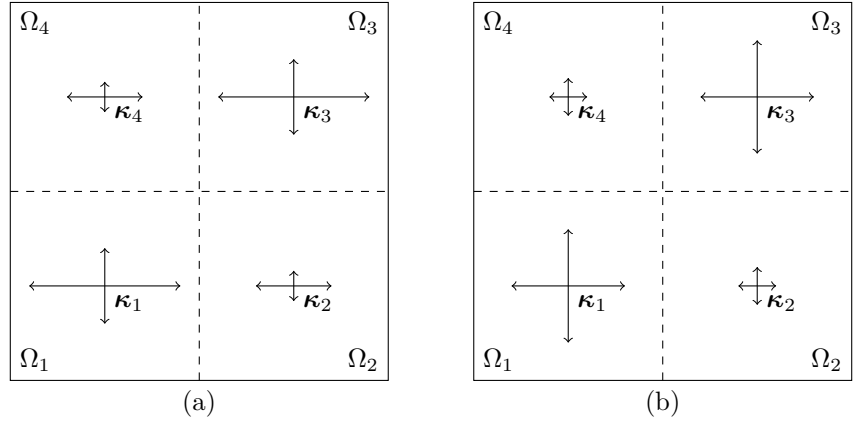

\paragraph{Convergence}
Figure~\ref{fig:test1-conv} reports the convergence history for $k=1,\dots,4$ at $\lambda=10^{3}$ (contrast $10^{6}$), in the energy norm (left) and the $\LL{2}$-norm (right). The energy curves approach the order $\bigO{h^k}$ predicted by Theorem~\ref{thm:nrj-estimate}. In the energy norm the three variants S-, N- and I-IPN are numerically indistinguishable, their errors agreeing to within $0.25\%$ across the whole sweep. Thus no appreciable dependence of the energy error on $\epsilon$ is observed on this benchmark. The $\LL{2}$ curves are compatible with order $k+1$ for all three variants, with some pre-asymptotic slopes exceeding this value; S-IPN has the smallest errors. These observations are not consequences of Theorem~\ref{thm:L2-estimates}, whose adjoint regularity assumption has not been verified for this medium. Under that assumption, the theorem would guarantee order $k+1$ for S-IPN and $k+1/2$ for the two other variants.

\begin{figure}[!ht]
\centering
\includegraphics[width=0.49\textwidth]{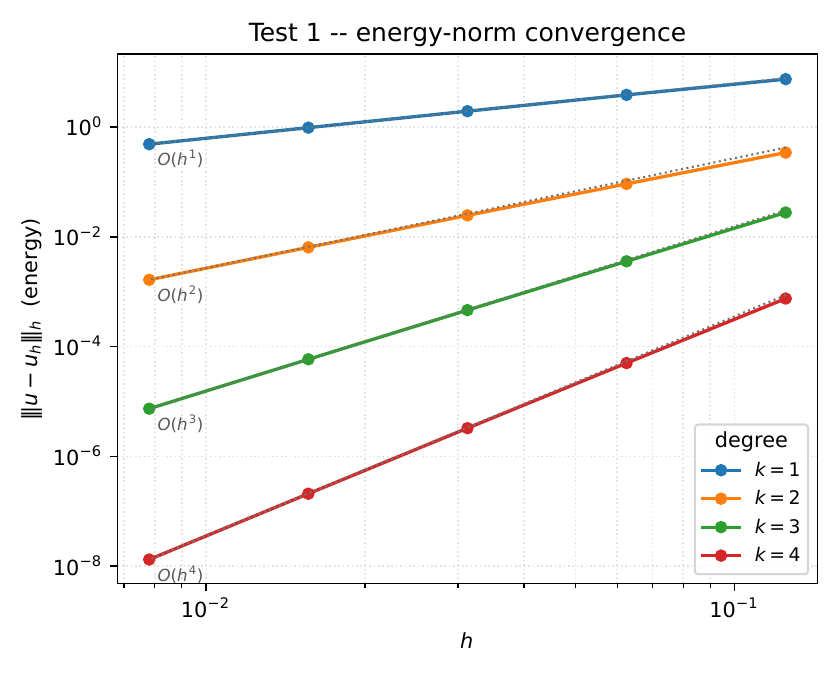}\hfill
\includegraphics[width=0.49\textwidth]{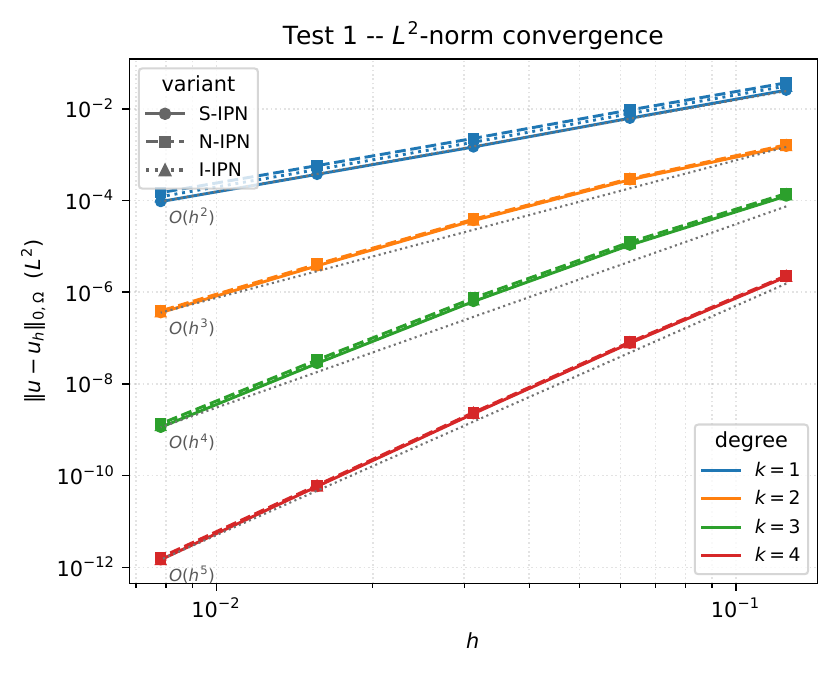}
\caption{Test~1 (anisotropic checkerboard, $\lambda=10^{3}$) --- convergence of the energy error (left) and the $\LL{2}$-error (right) versus $h$, for $k=1,\dots,4$. The three variants $\{$S,N,I$\}$-IPN coincide in the energy norm; in the $\LL{2}$-norm the curves are compatible with order $k+1$, with the smallest errors for S-IPN.}
\label{fig:test1-conv}
\end{figure}

\paragraph{Contrast robustness}
Fixing $h=1/64$ and $k=2$, we let $\lambda$ range over twelve orders of contrast. Figure~\ref{fig:test1-robust} reports the relative energy error $\tnorm{u-\du}_{h}/\norm{\bkap^{\ohalf}\grd u}{0}{\domain}$ of the three variants: it stays within $[1.5,\,2.4]\times10^{-4}$ across $\lambda\in[1,10^6]$ (contrast up to $10^{12}$), i.e.\ essentially flat---and even mildly \emph{decreasing} with the contrast. The three curves are moreover indistinguishable (within $0.09\%$), so the robustness is shared by S-, N- and I-IPN alike. For this family of smooth solutions, the measured relative accuracy is therefore insensitive to the contrast over the tested range; the discrete solution remains clean and oscillation-free there (Figure~\ref{fig:solutions}(a)).

\begin{figure}[!ht]
\centering
\includegraphics[width=0.62\textwidth]{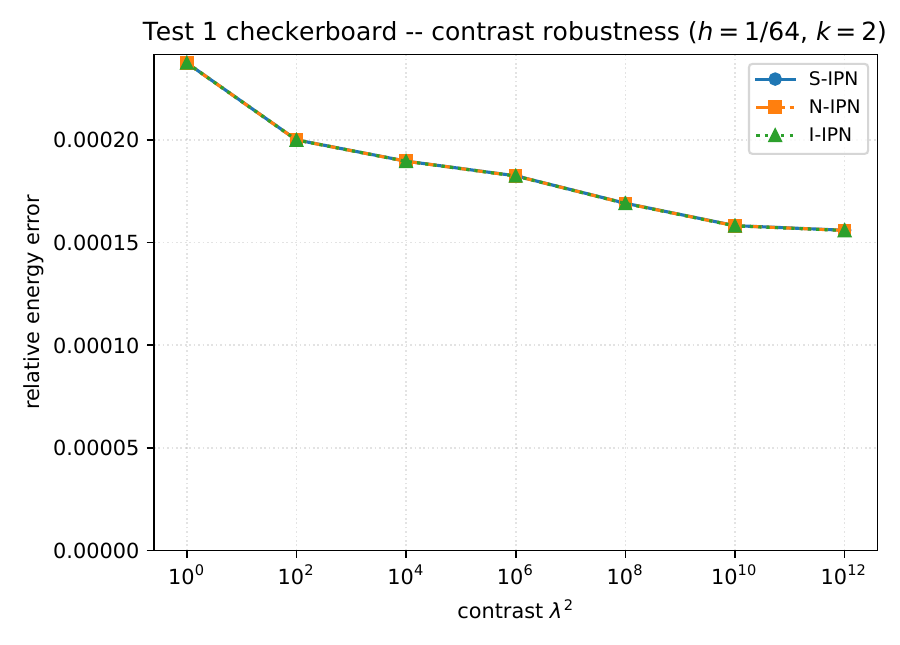}
\caption{Test~1 (anisotropic checkerboard) --- relative energy error versus the contrast $\lambda^2$ ($h=1/64$, $k=2$) for the three variants: the error remains flat over twelve orders of magnitude, identically for S-, N- and I-IPN, including in the presence of the checkerboard cross-point.}
\label{fig:test1-robust}
\end{figure}

\begin{figure}[!ht]
\centering
\includegraphics[width=0.94\textwidth]{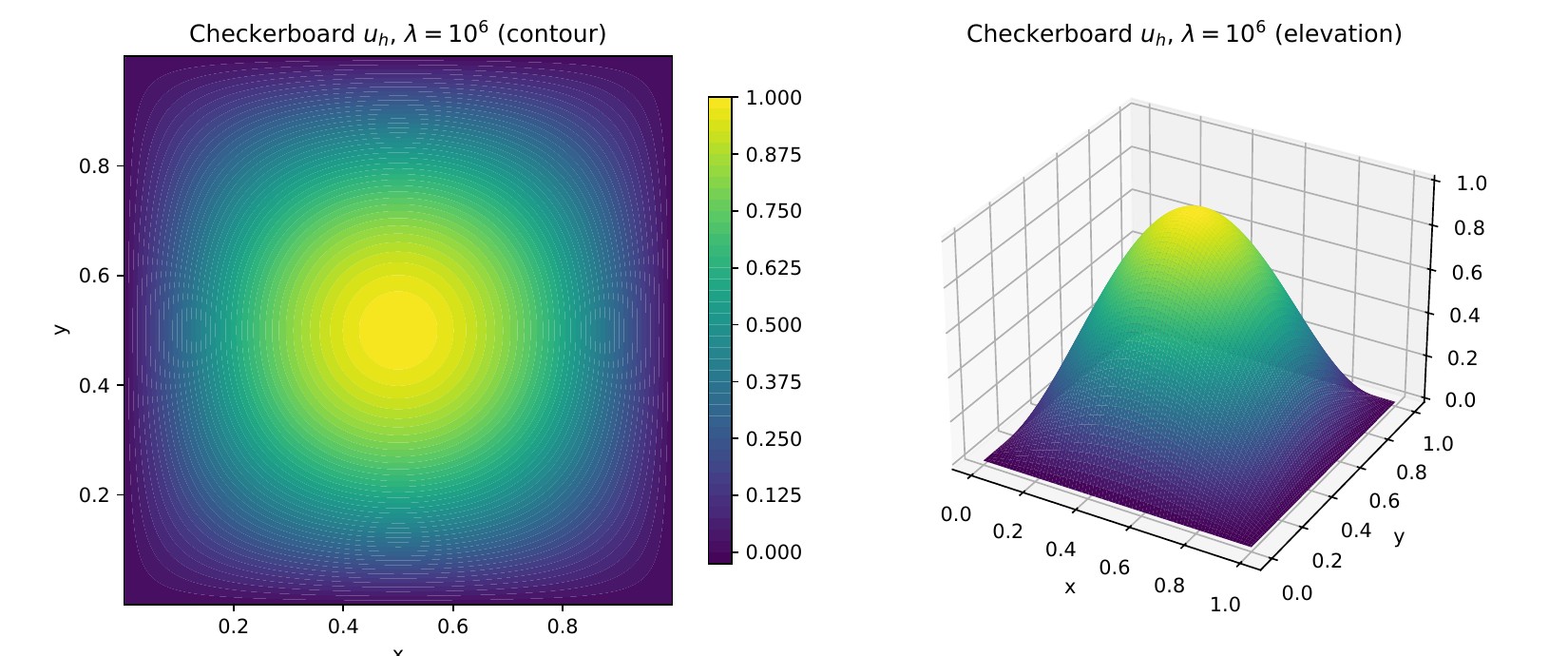}\\(a)

\medskip
\includegraphics[width=0.94\textwidth]{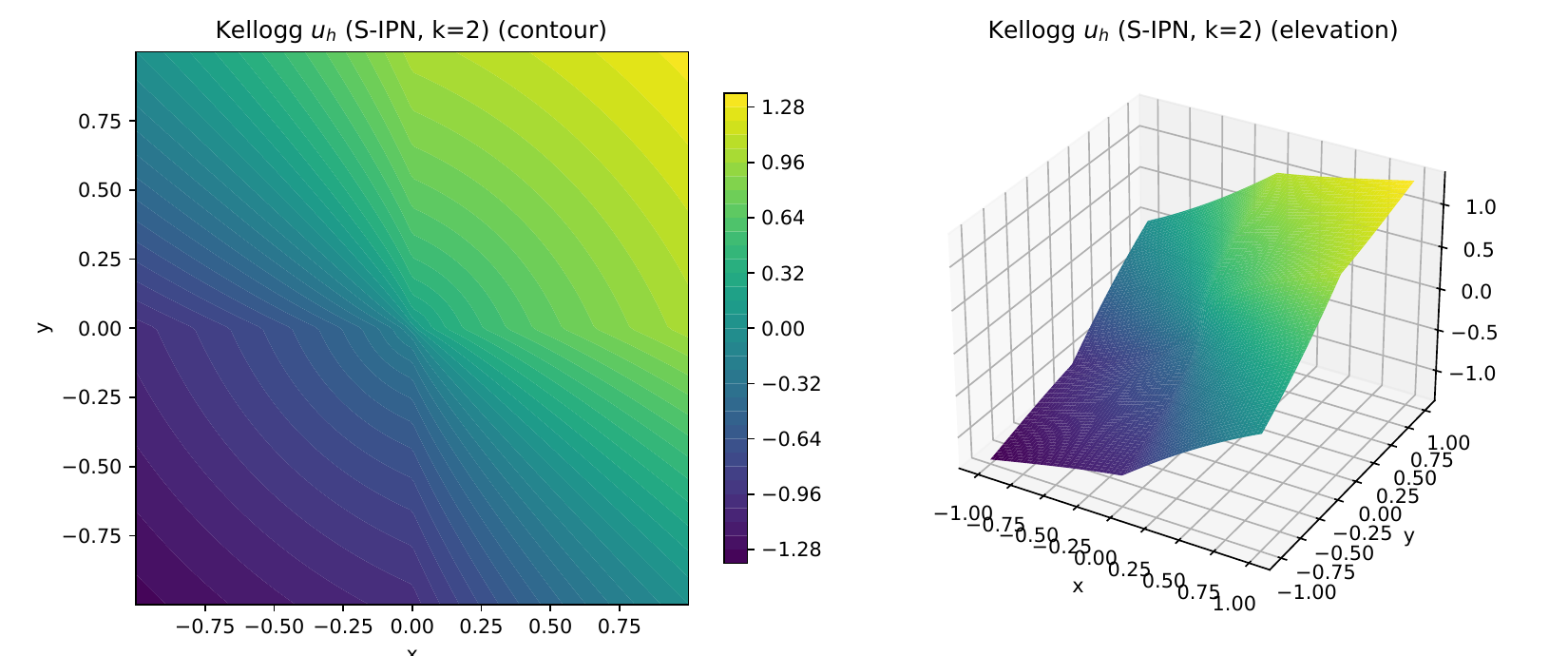}\\(b)
\caption{Discrete solutions $\du$ (S-IPN) on the four-subdomain media. (a) Test~1 at $\lambda=10^{6}$: clean and oscillation-free, including at the cross-point $(\tfrac12,\tfrac12)$. (b) Test~2 (Kellogg): the point singularity at the origin.}
\label{fig:solutions}
\end{figure}

\subsection{Test 2: heterogeneous diffusion with a singular solution (Kellogg)}
\label{S:5.2}

We turn to the Kellogg benchmark~\cite{Kellogg1975,Riviere08bookDG}, which probes the low-regularity regime. On $\domain=(-1,1)^2$, partitioned into the same four subdomains (Figure~\ref{fig:test1-geometry}(b)) with isotropic tensor $\bkap\vert_{\domain_i}=\kappa_i\vect I$ ($\kappa_1=\kappa_3=5$, $\kappa_2=\kappa_4=1$), the exact solution $u(r,\vartheta)=r^{\gamma}\mu(\vartheta)$ has piecewise regularity $H^{1+s}$ for every $s<\gamma$, but not at $s=\gamma$, where $\gamma\approx0.5354$~\cite{Kellogg1975} and exhibits a point singularity at the origin (Figure~\ref{fig:solutions}(b)), enforced through inhomogeneous Dirichlet data (Remark~\ref{rmk:inhomogeneous}). The convergence is then governed by the regularity of $u$ rather than by the polynomial degree, and the energy estimate gives $\bigO{h^s}$ for every fixed $1/2<s<\gamma$ and every $k\geq1$, with the weak boundary treatment of Remark~\ref{rmk:inhomogeneous}. It does not assert the endpoint rate $h^\gamma$. The measured energy rate is nevertheless close to $\gamma$: Table~\ref{tab:test2} and Figure~\ref{fig:test2} show the representative degree $k=2$. The archived sweep over $k=1,\dots,4$ exhibits the same limiting slope for the three variants, with higher degrees lowering the error constant. As in Test~1, the variants coincide in energy and share the $\LL{2}$ rate, the symmetric S-IPN having the smallest $\LL{2}$ constant; we therefore report the IPN method, the values being those of S-IPN. The observed $\LL{2}$-rate, $\approx1.1$, is about $2\gamma$: the adjoint operator has the same interface singular exponents as the primal one, so this empirical rate is consistent with a regularity-limited duality gain (Remark~\ref{rmk:dual-regularity}) and the rate~\eqref{L2-estimate_rate} does not apply, \eqref{elliptic_regularity} being violated here.

\begin{table}[!ht]
\centering
\caption{Test~2 (Kellogg, $\gamma\approx0.5354$), degree $k=2$: energy and $\LL{2}$ errors and observed rates of the IPN method (values for the representative S-IPN). The energy rate at this degree is close to $\gamma$ for each variant $\epsilon\in\{0,\pm1\}$; the three variants coincide in energy and share the $\LL{2}$ rate (Figure~\ref{fig:test2}).}
\label{tab:test2}
\begin{tabular}{c c c c c}
\hline
$1/h$ & $\tnorm{u-\du}_{h}$ & rate & $\norm{u-\du}{0}{\domain}$ & rate\\
\hline
$8$ & 2.63e-01 & --- & 6.46e-03 & ---\\
$16$ & 1.81e-01 & 0.54 & 2.96e-03 & 1.12\\
$32$ & 1.25e-01 & 0.54 & 1.38e-03 & 1.10\\
$64$ & 8.60e-02 & 0.54 & 6.52e-04 & 1.09\\
\hline
\end{tabular}
\end{table}

\begin{figure}[!ht]
\centering
\includegraphics[width=0.6\textwidth]{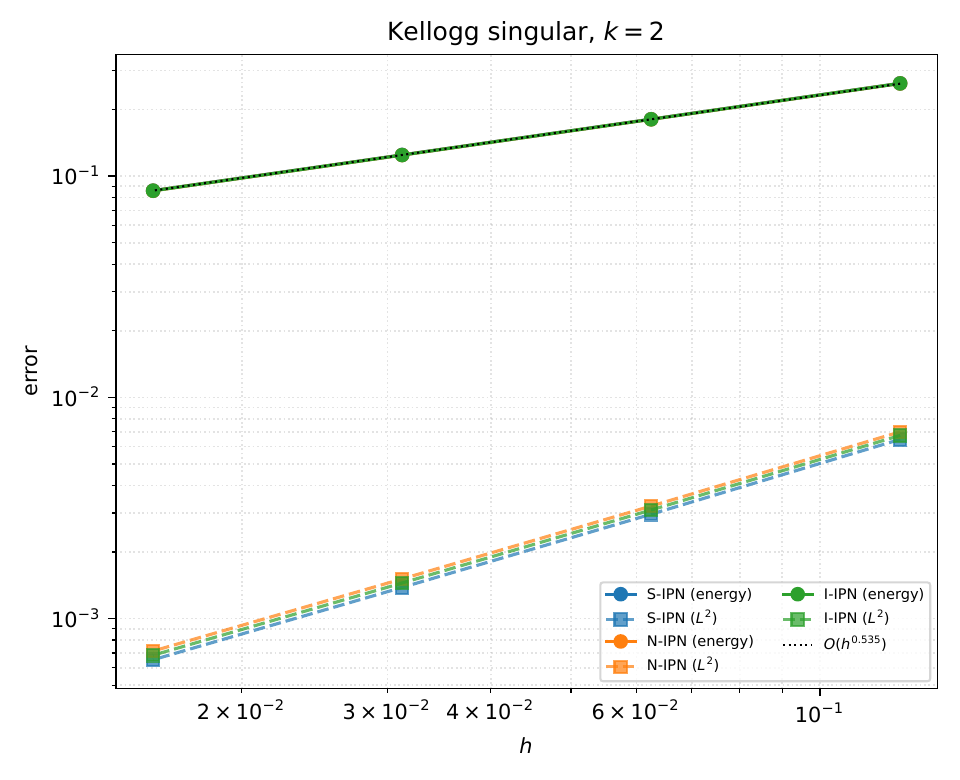}
\caption{Test~2 (Kellogg): convergence history at $k=2$ for the three variants $\{$S,N,I$\}$-IPN, exhibiting an observed energy slope close to $\gamma$; the variants coincide in energy.}
\label{fig:test2}
\end{figure}

\subsection{The conforming variant: convergence and the role of the transmissibility weighting}
\label{S:5.3}

This section is set in the conforming space~\eqref{CCIP_formulation}, where $\eta_{h}$ vanishes and the interior penalty $s_h=-\dualprod{\varrho\jump{\bkap\grdh\du}}{\jump{\bkap\grdh\dv}}{0}{\Fhi}$ is the only remaining mechanism; its penalized variant $\epsilon=-1$ is a transmissibility-weighted CIP method (Proposition~\ref{prop:conforming}). The boundary treatment can therefore play no part; moreover the mean weights $\vomega$ do not appear in $a^{\mathfrak c}_h$, so $\varrho$ is the only trace left by $\tau$ in the method, and a \emph{controlled} comparison isolates it exactly: the \emph{same} penalty term, with either the \emph{transmissibility} weight built on~\eqref{penalty_term}, $\tau_i\eqbydef\alpha_0\C{T}^2\,\kappa_{i,\f{}}\abs{\f{}}/\abs{\e{i}}$, or a \emph{contrast-blind} one obtained by removing the diffusion from $\tau$, $\tau^{\mathrm{cb}}_i\eqbydef\alpha_0\C{T}^2\abs{\f{}}/\abs{\e{i}}$, that is,
\begin{equation*}
\varrho^{\mathrm{tr}}=\frac{1}{\tau_1+\tau_2},\qquad
\varrho^{\mathrm{cb}}=\frac{1}{\tau^{\mathrm{cb}}_1+\tau^{\mathrm{cb}}_2}=\kappa_{\f{}}\,\varrho^{\mathrm{tr}},\qquad
\kappa_{\f{}}\eqbydef\frac{\kappa_{1,\f{}}/\abs{\e{1}}+\kappa_{2,\f{}}/\abs{\e{2}}}{1/\abs{\e{1}}+1/\abs{\e{2}}}.
\end{equation*}
What the comparison varies is therefore the \emph{scale} of $\varrho$. The jump penalized by $s_h$ is that of the flux $\bkap\grdh\du$, so $\varrho$ must carry $\kappa^{-1}$ for $s_h$ to have the dimension of the energy. Classical CIP methods typically penalize derivative jumps~\cite{DouglasCIP1976,BurmanHansbo2004}; here both branches penalize the same \emph{flux} jump. Thus the comparison is an ablation of the diffusion scaling within the IPN flux-jump term, not a comparison with a complete classical CIP method. The contrast-blind branch multiplies that penalty by $\kappa_{\f{}}$ relative to the transmissibility branch, and increases it where $\kappa_{\f{}}>1$. All else is fixed ($k=2$, checkerboard).

The transmissibility branch is the penalized conforming variant itself: at $\lambda=10^{3}$, refinement from $h=1/64$ to $h=1/128$ gives the observed orders $1.96$ in the energy norm and $3.39$ in the $\LL{2}$-norm, showing energy convergence close to order two and a pre-asymptotic $\LL{2}$ slope above three. As in Section~\ref{S:5.1}, the $\LL{2}$ observation is empirical: Remark~\ref{rmk:L2-conforming} requires the adjoint regularity assumption, which is not established on this checkerboard.

Table~\ref{tab:weighting} shows that the energy-error ratio approaches one at the lowest tested contrast, whereas it increases over the available refinement range at the two higher contrasts. At contrast $10^8$, the contrast-blind weighting produces an energy error $2.56$ times larger on the finest mesh. The transmissibility weighting maintains observed energy orders between $1.93$ and $2.00$, while the contrast-blind branch has not reached order two at the highest contrast.

These results indicate a contrast-dependent delay in attaining the expected energy convergence with the contrast-blind weighting. Both stabilizations are exactly consistent (Remark~\ref{rmk:exact-consistency}), but this concerns the exact solution and does not eliminate their action on the discrete solution. The comparison isolates the interior flux-jump mechanism; it does not by itself explain the robustness of the complete IPN formulation, particularly I-IPN, whose interior stabilization is absent.

\begin{table}[!ht]
\centering
\caption{Effect of the interior-penalty weighting (checkerboard, $k=2$, penalized variant): ratio of the energy error of the contrast-blind flux-jump weight to that of the transmissibility weight, versus the contrast $\lambda^2$ and the mesh size. The ratio decreases towards one for $\lambda=10^2$ and increases for $\lambda=10^3$ and $10^4$; no turnover is observed for these latter two contrasts on the available meshes.}
\label{tab:weighting}
\begin{tabular}{c c c c c c}
\hline
$\lambda$ & contrast $\lambda^2$ & $h=1/32$ & $1/64$ & $1/128$ & $1/256$\\
\hline
$10^{2}$ & $10^{4}$ & 1.24 & 1.19 & 1.12 & 1.06\\
$10^{3}$ & $10^{6}$ & 1.38 & 1.91 & 2.53 & 2.67\\
$10^{4}$ & $10^{8}$ & 1.13 & 1.20 & 1.57 & 2.56\\
\hline
\end{tabular}
\end{table}

\section{Conclusion}

We have constructed and analyzed an Interior Penalty Nitsche (IPN) method for diffusion with heterogeneous, anisotropic coefficients. The method is obtained by algebraically condensing a stabilized Lagrange multiplier defined on the whole mesh skeleton, which yields an $H^1$-conforming formulation combining a Nitsche boundary penalty with a transmissibility-weighted interior penalty on the jump of the normal diffusive flux. We established the equivalence between the hybridized and condensed formulations, which delivers the well-posedness of the hybridized problem as a by-product, and we proved discrete coercivity, boundedness, and \textit{a priori} error estimates of the scalar variable in the energy- and $\LL{2}$-norms for the three symmetry variants. The continuity of the numerical normal flux across interfaces is available in closed form and inherited from the hybridization strategy (Remark~\ref{rmk:conservation}). The energy approximation estimate uses elementwise interpolation and requires no quasi-monotonicity of the diffusion coefficients. The benchmarks exhibit the expected energy convergence and stable relative accuracy over the tested contrasts. Their $\LL{2}$ results complement the theory empirically, since the piecewise adjoint regularity required by the $\LL{2}$ estimate is not established for the smooth checkerboard and fails for the Kellogg configuration.

The strongly imposed boundary variant exposes the CG / stabilized / sign-reversed trichotomy analyzed above. Its controlled numerical comparison isolates the role of transmissibility weighting: over the tested meshes and contrasts, removing the diffusion scaling delays the observed energy convergence, while the transmissibility weighting retains rates close to the optimal order.

A direction for future work is to use the single-valued normal trace $\dtsg$ as a starting point for an $\boldsymbol{\mathrm{H}}(\mathrm{div})$-conforming, locally conservative flux reconstruction and a guaranteed Prager--Synge a posteriori estimator. Local equilibration and the robustness of such an estimator require additional analysis. The extension of the IPN construction to linear elasticity is developed in a separate companion manuscript, which treats matrix-valued stabilization and the transmission of numerical tractions.

\section*{Reproducibility}
The code and reference data accompanying this article are available on Zenodo at \url{https://doi.org/10.5281/zenodo.23524954}. The archive includes the Python/NGSolve implementations, benchmark drivers, figure-generation scripts, and instructions for reproducing the numerical results.

\section*{Declaration of competing interest}
The author declares no competing interests.

\bibliography{ipn_ref}

\end{document}